\documentclass[11pt]{amsart} \textwidth=14.5cm \oddsidemargin=1cm
\usepackage{amsmath, amsxtra, amscd, amsthm, amsfonts, amssymb, eucal, bbding, wasysym}
\usepackage{graphics, color}
\usepackage{hyperref}

\input prepictex
\input pictex
\input postpictex
\newtheorem{thm}{Theorem}[section]
\newtheorem{lem}[thm]{Lemma}
\newtheorem{cor}[thm]{Corollary}
\newtheorem{prop}[thm]{Proposition}

\theoremstyle{definition}

\newtheorem{example}[thm]{Example}
\newtheorem{conjecture}[thm]{Conjecture}

\theoremstyle{remark}
\newtheorem{rem}[thm]{Remark}

\numberwithin{equation}{section}

\begin{document}

%Referring commands:
\newcommand{\thmref}[1]{Theorem~\ref{#1}}
\newcommand{\secref}[1]{Section~\ref{#1}}
\newcommand{\lemref}[1]{Lemma~\ref{#1}}
\newcommand{\propref}[1]{Proposition~\ref{#1}}
\newcommand{\corref}[1]{Corollary~\ref{#1}}
\newcommand{\remref}[1]{Remark~\ref{#1}}
\newcommand{\eqnref}[1]{(\ref{#1})}
\newcommand{\exref}[1]{Example~\ref{#1}}

%Simplified symbols:
\newcommand{\nc}{\newcommand}
\nc{\Z}{{\mathbb Z}}
\nc{\C}{{\mathbb C}}
\nc{\N}{{\mathbb N}}
\nc{\F}{{\mf F}}
\nc{\Q}{\ol{Q}}
\nc{\la}{\lambda}
\nc{\ep}{\epsilon}
\nc{\h}{\mathfrak h}
\nc{\st}{\texttt{st}}
\nc{\n}{\mf n}
\nc{\A}{{\mf a}}
\nc{\G}{{\mathfrak g}}
\nc{\SG}{\overline{\mathfrak g}}
\nc{\DG}{\widetilde{\mathfrak g}}
\nc{\D}{\mc D} \nc{\Li}{{\mc L}} \nc{\La}{\Lambda} \nc{\is}{{\mathbf
i}} \nc{\V}{\mf V} \nc{\bi}{\bibitem} \nc{\NS}{\mf N}
\nc{\dt}{\mathord{\hbox{${\frac{d}{d t}}$}}} \nc{\E}{\mc E}
\nc{\ba}{\tilde{\pa}} \nc{\half}{\frac{1}{2}} \nc{\mc}{\mathcal}
\nc{\mf}{\mathfrak} \nc{\hf}{\frac{1}{2}}
\nc{\hgl}{\widehat{\mathfrak{gl}}} \nc{\gl}{{\mathfrak{gl}}}
\nc{\hz}{\hf+\Z}
\nc{\dinfty}{{\infty\vert\infty}} \nc{\SLa}{\overline{\Lambda}}
\nc{\SF}{\overline{\mathfrak F}} \nc{\SP}{\overline{\mathcal P}}
\nc{\U}{\mathfrak u} \nc{\SU}{\overline{\mathfrak u}}
\nc{\ov}{\overline}
\nc{\wt}{\widetilde}
\nc{\wh}{\widehat}
\nc{\sL}{\ov{\mf{l}}}
\nc{\sP}{\ov{\mf{p}}}
\nc{\osp}{\mf{osp}}
\nc{\spo}{\mf{spo}}
\nc{\even}{{\bar 0}}
\nc{\odd}{{\bar 1}}
\nc{\I}{\mathbb{I}}
\nc{\X}{\mathbb{X}}
\nc{\hh}{\widehat{\mf{h}}}
\nc{\Icirc}{I_{\text{\,\begin{picture}(2,2)\setlength{\unitlength}{0.07in} \put(0.3,0.45){\circle{1}}\end{picture}}}\,}
\nc{\Ibullet}{I_{\text{\,\begin{picture}(2,2)\setlength{\unitlength}{0.07in} \put(0.3,0.45){\circle*{1}}\end{picture}}}\,}

\newcommand{\blue}[1]{{\color{blue}#1}}
\newcommand{\red}[1]{{\color{red}#1}}
\newcommand{\green}[1]{{\color{green}#1}}
\newcommand{\white}[1]{{\color{white}#1}}

 \advance\headheight by 2pt

\title[Kac-Wakimoto formula for ortho-symplectic Lie superalgebras]{Kac-Wakimoto character formula for ortho-symplectic Lie superalgebras}

\author[Cheng]{Shun-Jen Cheng$^\dagger$}
\thanks{$^\dagger$Partially supported by a MoST and an Academia Sinica Investigator grant}
\address{Institute of Mathematics, Academia Sinica, Taipei,
Taiwan 10617} \email{chengsj@math.sinica.edu.tw}

\author[Kwon]{Jae-Hoon Kwon$^{\dagger\dagger}$}
\thanks{$^{\dagger\dagger}$Partially supported by an NRF-grant 2011-0006735.}
\address{Department of Mathematical Sciences, Seoul National University, Seoul 151-747, Korea
 }
\email{jaehoonkw@snu.ac.kr}

\keywords{ortho-symplectic Lie superalgebras, Kac-Wakimoto formula}

\begin{abstract} We classify finite-dimensional tame modules over the ortho-symplectic Lie superalgebras. For these modules we show that their characters are given by the Kac-Wakimoto character formula, thus establishing the Kac-Wakimoto conjecture for the ortho-symplectic Lie superalgebras. We further relate the Kac-Wakimoto formula to the super Jacobi polynomials of Sergeev and Veselov, and show that these polynomials, up to a sign, give the super characters for these tame modules.
\end{abstract}
%\noindent{\bf Key words:} ortho-symplectic Lie superalgebras, Kac-Wakimoto formula.

%\vspace{.3cm}

%\noindent{\bf Mathematics Subject Classifications (1991)}: 17B67.

\maketitle

\section{Introduction}

The characters of the finite-dimensional irreducible modules over a finite-dimensional simple Lie algebra is given by the celebrated Weyl character formula. This character formula is known to generalize to the so-called Weyl-Kac character formula for integrable highest weight modules over Kac-Moody Lie algebras (see e.g., \cite{K3}). However, such a formula does not hold for finite-dimensional irreducible modules over finite-dimensional classical Lie superalgebras, in general. The only class of finite-dimensional irreducible modules for which such a character formula holds is the class of the so-called typical modules \cite{K2} (see Section \ref{subsec:general} for a definition).

For the finite-dimensional irreducible modules over the general linear Lie superalgebras of degree of atypicality $1$ (see Section \ref{subsec:general} for a definition), Bernstein and Leites \cite{BL} produced a character formula, that has resemblance with the Weyl character formula. For $\osp(2|2n)$, an analogous formula was then established by van der Jeugt in \cite{vdJ}. However, it was quickly realized that the Bernstein-Leites' formula does not generalize to arbitrary finite-dimensional irreducible modules of degrees of atypicality exceeding one.

There is another character formula of closed form, which resembles, and in a way generalizes, the Bernstein-Leites formula, often referred to as the {\em Kac-Wakimoto formula or conjecture} \cite{KW1,KW2}.
%Now, the Kac-Wakimoto character formula, which is a closed formula, resembles, and in a way generalizes, the Bernstein-Leites formula.
Just like the Bernstein-Leites formula, the Kac-Wakimoto formula does not hold for arbitrary finite-dimensional irreducible modules over basic Lie superalgebras.  Indeed, the conjecture of Kac and Wakimoto \cite[Conjecture 3.6]{KW2} (see also Conjecture \ref{conj:KW}) makes explicit a (conjectural) subset of finite-dimensional irreducible modules whose characters are given by the Kac-Wakimoto character formula when the degree of atypicality may exceed one.

The finite-dimensional irreducible characters over the general linear Lie superalgebras was first computed by Serganova in \cite{Ser}, where an algorithm was given for their Kazhdan-Lusztig polynomials. Brundan then related these polynomials to Lusztig's canonical bases \cite{Br}. Gruson and Serganova in \cite{GS} gave an algorithm to compute the finite-dimensional irreducible characters of the ortho-symplectic Lie superalgebras. Irreducible characters for modules over these basic Lie superalgebras in certain parabolic and then the general BGG categories are then obtained in \cite{CL, CLW, CLW2, BW, Br}. From these works, it is evident that even the finite-dimensional irreducible characters are controlled by Kazhdan-Lusztig type polynomials that are in general rather complicated to compute. Thus, closed formulas are not expected in general.

It is therefore surprising that there should be a large class of {\em tame} modules (see Section \ref{sec:tame:KW}), which includes the typical modules as the simplest subclass, over basic Lie superalgebras, and whose characters should be given by certain simple closed formulas, as predicted by the Kac-Wakimoto conjecture. For the general linear Lie superalgebra the Kac-Wakimoto conjecture has recently been verified by Chmutov, Hoyt, and Reif \cite{CHR}. In loc.~cit.~the tame modules were shown to be precisely the Kostant modules of Brundan and Stroppel \cite{BS}. The Kac-Wakimoto conjecture for the general linear Lie superalgebra is then derived in \cite{CHR} using a character formula for Kostant modules given in \cite{SZ1}.

The finite-dimensional representation theory of the exceptional Lie superalgebra $D(2|1,\alpha)$, $G(3)$, and $F(3|1)$ is simplified by the fact that the degree of atypicality of the irreducible modules is at most one. The finite-dimensional irreducible modules afford Bernstein-Leites type character formulas, see \cite[Theorem 7.1(1)]{SZ2}. Note that finite-dimensional irreducible characters of $D(2|1,\alpha)$ were also calculated in \cite[Section 3]{Ger}, while for $G(3)$ and $F(3|1)$, similar formulas were also obtained in \cite[Theorem 2.6]{Mar}. Since the rank of these Lie superalgebras are small, the conjugacy classes of Borel subalgebras are readily classified (besides they are certainly well-known), from which one obtains a classification of their tame modules in a fairly straightforward manner.  Once the tame modules have been identified, one then can derive the Kac-Wakimoto conjecture for these exceptional Lie superalgebras directly from the formulas in \cite[Theorem 7.1(1)]{SZ2}. We shall not give details in the present paper, as the approach is different and also easier. Besides, we think that experts may be aware of this fact.

Thus, the remaining open, and technically most demanding case of the Kac-Wakimoto conjecture for basic Lie superalgebras (Conjecture \ref{conj:KW}) is the case of the ortho-symplectic Lie superalgebras. The main purpose of this paper is to settle this conjecture in this case (Theorem \ref{thm:main result}). We shall use the remainder of this introduction to give an outline of the proof, and, at the same time,  also explain the organization of the paper.

In Section \ref{sec:prelim} we set up notations and collect some preliminary results on ortho-symplectic Lie superalgebras that are used in the sequel. The classification of Borel subalgebras for the ortho-symplectic Lie superalgebras are recalled, along with basic facts about Zuckermann's cohomological induction in the setting of Lie superalgebras following Germoni, Santos, and Serganova \cite{Ger, San, Ser}. Section \ref{sec:prelim} then concludes with a recollection of Gruson and Serganova's ``typical lemma'' (Lemma \ref{lem:GS}).

In Section \ref{sec:tame:KW} we classify the highest weights of tame modules over the ortho-symplectic Lie superalgebras. We simplify this task by first showing that in the classification it is enough to consider conjugacy classes of Borel subalgebras instead of all Borel subalgebras. As is usual, when studying the ortho-symplectic Lie superalgebra $\osp(\ell|2n)$, it is necessary to study the cases of $\ell$ even and odd separately. We classify the highest weights of tame modules for $\osp(2m+1|2n)$ and $\osp(2m|2n)$ in Theorems \ref{thm:classification B} and \ref{thm:classification D}, respectively. It turns out that such highest weights, if they are not typical, are precisely the highest weights that have only atypicality of ``type $A$'' (cf.~Lemma \ref{lem:T:set:Atype}). We also show in Proposition \ref{prop:Euler=KW} that the validity of the Kac-Wakimoto conjecture for the trivial modules implies that the characters of certain Euler characteristics are precisely given by the Kac-Wakimoto character formula.

In Section \ref{sec:hw:tame} highest weights of tame modules are studied in detail.  A main result in this section is Theorem \ref{thm:KW:bottom} which says that the highest weights of tame modules are special in the following sense: Let $\G=\osp(\ell|2n)$ with $\ell\ge 3$. Consider the standard Borel subalgebra (see Section \ref{sec:osp def}) and denote the set of positive roots by $\Phi^+$. Denote by $\preccurlyeq $ the usual partial ordering on weights induced by the positive root lattice with respect to $\Phi^+$. A particular feature of the finite-dimensional representation theory of $\osp(\ell|2n)$, for $\ell\ge 3$, is that there are highest weights that are minimal with respect to the partial order $\preccurlyeq$. We show that these ``bottoms'' of this partial ordering are essentially the highest weights of the tame modules. To be more precise, this is indeed the case when $\ell$ is odd. For $\ell$ even, it turns out that this is the case when the degree of atypicality exceeds one. However, in the case when the degree of atypicality equals one, this is not always true.  What remains true is that the ``bottoms'' are all highest weights of tame modules (Theorem \ref{thm:KW:bottom}). We prove the technical Proposition \ref{prop:aux01}, which allows us to make use of the ``typical'' lemma in Section \ref{sec:proof:KW}.

In Section \ref{sec:proof:KW} the Kac-Wakimoto conjecture for ortho-symplectic Lie superalgebras is proved in Theorem \ref{thm:main result}. We first establish the conjecture for the case of the trivial module in Proposition \ref{prop:KW:trivial}.  We then prove that the virtual modules in the Euler characteristics of Proposition \ref{prop:Euler=KW} above are indeed irreducible modules themselves, which then completes the proof. It is proved first with respect to a special class of Borel subalgebras $\mf b^{\texttt{odd}}$ (cf.~\cite{GS}), and then extended to the general case. We conclude the paper by showing that the Kac-Wakimoto super character formula are, up to a sign, essentially the Sergeev-Veselov's specialized super Jacobi polynomials in Corollary \ref{cor:super:Jacobi}. These polynomials were studied in \cite{SV1} in the context of deformed Calogero-Moser systems.

\vspace{.2cm} \noindent {\bf Acknowledgment.}
The second author thanks the Institute of Mathematics, Academia Sinica, Taipei, for hospitality and support. After our paper has been completed, the first author was informed by Victor Kac that he and Maria Gorelik just finished the paper ``Characters of (relatively) integrable modules over affine Lie superalgebras" (arXiv:1406.6860), where Conjecture \ref{conj:KW} for ortho-symplectic Lie superalgebra has been proved in most cases using different method. From their paper we learned that a factor of $2$ was missing in $j_\la$ in some cases in the first version of our manuscript. We thank Shifra Reif and Victor Kac for correspondence. Finally, we wish to thank the anonymous referee for useful comments and suggestions.

\vspace{.2cm} \noindent {\bf Notation.}
We assume that our base field is $\mathbb{C}$. All algebras, vector spaces, etc., are over $\C$.

\section{Preliminaries}\label{sec:prelim}

\subsection{Generalities on basic Lie superalgebras}\label{subsec:general}

Let $\G$ stand for a finite-dimensional basic Lie superalgebra. Suppose that $\mf b$ is a Borel subalgebra of $\G$ containing a Cartan subalgebra $\h$. We denote the set of simple and positive roots corresponding to $\mf b$ by $\Pi_{\mf b}$ and $\Phi^+_{\mf b}$, respectively. Furthermore, we let $\Phi^+_{{\mf b},\bar{0}}$ and $\Phi^+_{{\mf b},\bar{1}}$ stand for positive even and odd roots, respectively, and let
\begin{equation*}
\rho^{\mf b}_{\even}:=\frac{1}{2}\sum_{\alpha\in\Phi^+_{\mf b,\bar 0}}\alpha,\quad \rho^{\mf b}_{\odd}:=\frac{1}{2}\sum_{\beta\in\Phi^+_{\mf b,\bar 1}}\beta,\quad \rho^{\mf b}:=\rho^{\mf b}_{\even}-\rho^{\mf b}_{\odd}.
\end{equation*}
Then $\rho^{\mf b}$ is the corresponding Weyl vector.
Furthermore, for an indeterminate $e$, we set
\begin{align*}
D_{\mf b,\bar 0}:=\prod_{\alpha\in\Phi^+_{\mf b,\bar 0}}(e^{\alpha/2}-e^{-\alpha/2}),\quad
D_{\mf b, \bar 1}:=\prod_{\beta\in\Phi^+_{\mf b,\bar 1}}(e^{\beta/2}+e^{-\beta/2}),\quad
D_{\mf b}:=\frac{D_{\mf b, \bar 0}}{D_{\mf b, \bar 1}}.
\end{align*}
Note that for different $\mf b$'s the various expressions for $D_{\mf b, \bar 0}$ differ possibly by a sign, while the expression $D_{\mf b, \bar 1}$ is independent of $\mf b$. The subscript $\mf b$ will be dropped when there is no danger of confusion.

We denote by $W_{\G}$, or simply $W$ if there is no confusion, the Weyl group of $\G$, which is by definition the Weyl group of the reductive Lie algebra $\G_\even$. The length of an element $w\in W$ is denoted by $\ell(w)$.

For a Lie superalgebra $\G$ with Borel and Cartan subalgebras $\mf b$ and $\h$, respectively, we denote by $L_{\mf b}(\G,\gamma)$, the irreducible highest weight $\G$-module with $\mf b$-highest weight $\gamma\in\h^*$.
When $\mf b$ is clear from the context, we shall sometimes omit the subscript $\mf b$ and write $L(\G,\gamma)$.  When clear from the context, we may also omit $\G$, or both $\mf b$ and $\G$.

Recall that $\G$ has an invariant super-symmetric non-degenerate bilinear form $(\cdot,\cdot)$, which restricts to a non-generate bilinear form on $\h$, and which in turn induces a symmetric bilinear form on $\h^*$.  By abuse of notation we shall denote the form on $\h^*$ also by $(\cdot,\cdot)$. Recall that a root $\alpha$ is said to be {\em isotropic} if $(\alpha,\alpha)=0$.

%Fix a Borel subalgebra $\mf b_\even$ of the even subalgebra $\G_\even$.
Let $\mf b$ be a Borel subalgebra of $\G$.
Let $\gamma\in\h^*$, which we may regard as a $\mf b$-highest weight for an irreducible $\G$-module. Recall that the weight $\gamma$ is said to be {\em typical} if $(\gamma+\rho^{\mf b},\alpha)\not=0$, for all isotropic root $\alpha$. The corresponding irreducible highest weight  $\G$-module $L_{\mf b}(\G,\gamma)$ is then called {\em typical}. The module $L_{\mf b}(\G,\gamma)$ is said to be {\em atypical}, if it is not typical. In this case we may define the {\em degree of atypicality} of $L_{\mf b}(\G,\gamma)$ to be the maximal number of mutually orthogonal isotropic roots in $\Phi^+_{\mf b}$ that are orthogonal to $\gamma+\rho^{\mf b}$. We recall that the notions of typicality and degree of atypicality of an irreducible highest weight module are invariant under odd reflections  \cite[Section 1]{PS} (see also \cite[Section 1.4]{CW}), and hence they are independent of the Borel subalgebra $\mf b$ (cf.~Lemma \ref{lem:T:set:Atype}).

For a $\G$-module $V$ that is $\h$-semisimple, e.g., when $V$ is finite dimensional and irreducible, we shall denote by $V_\mu$ its $\mu$-weight space for $\mu\in\h^*$, that is, $$V_\mu=\{v\in V\,|\,hv=\mu(h)v \ \text{for } h\in\h \}.$$

\subsection{General linear Lie superalgebra $\gl(n|m)$}\label{sec:glnm}
For $m,n\ge 1$,
let $\C^{n|m}$ be the complex superspace of dimension $(n|m)$. Choose an ordered basis $\{v_1,\ldots,v_n\}$ for $\C^{n|0}$ and an ordered basis $\{w_1,\ldots,w_m\}$ of $\C^{0|m}$ so that the Lie superalgebra $\gl(n|m)$ can be realized as $(n+m)\times (n+m)$ complex matrices. We denote by $E_{ij}$ the elementary matrix with $(i,j)$-entry $1$ and other entries $0$. Then the subalgebra of diagonal matrices is the standard Cartan subalgebra $\underline{\h}$ with standard basis $\{E_{11},\ldots,E_{nn},E_{n+1\,n+1},\ldots,E_{n+m\,n+m}\}$ so that its dual $\underline{\h}^*$ has dual basis $\{\delta_1,\ldots,\delta_n,\epsilon_1,\ldots,\epsilon_m\}$. The standard Borel subalgebra $\underline{\mf b}^{\texttt{st}}$ of $\gl(n|m)$ is the subalgebra of upper triangular matrices and its corresponding Dynkin diagram is given by

\begin{equation}\label{ABC:diagram:A1}
\hskip 1cm \setlength{\unitlength}{0.16in}
\begin{picture}(24,1)
\put(-.5,0.3){\makebox(0,0)[c]{$\gl(n|m)$}}
\put(5.7,0.5){\makebox(0,0)[c]{$\bigcirc$}}
\put(8,0.5){\makebox(0,0)[c]{$\bigcirc$}}
\put(10.4,0.5){\makebox(0,0)[c]{$\cdots$}}
\put(12.5,0.45){\makebox(0,0)[c]{$\bigotimes$}}
\put(14.85,0.5){\makebox(0,0)[c]{$\bigcirc$}}
\put(17.1,0.5){\makebox(0,0)[c]{$\cdots$}}
\put(19.3,0.5){\makebox(0,0)[c]{$\bigcirc$}}
\put(6.1,0.5){\line(1,0){1.4}} \put(8.4,0.5){\line(1,0){1}}
\put(11,0.5){\line(1,0){1}} \put(13.1,0.5){\line(1,0){1.2}}
\put(15.28,0.5){\line(1,0){1}} \put(17.7,0.5){\line(1,0){1.1}}
\put(5.5,-0.5){\makebox(0,0)[c]{\tiny $\delta_1-\delta_{2}$}}
\put(8.2,-0.5){\makebox(0,0)[c]{\tiny $\delta_{2}-\delta_{3}$}}
\put(12.1,-0.5){\makebox(0,0)[c]{\tiny $\delta_n-\ep_1$}}
\put(14.8,-0.5){\makebox(0,0)[c]{\tiny $\ep_1-\ep_2$}}
\put(19,-0.5){\makebox(0,0)[c]{\tiny $\ep_{m-1}-\ep_m$}}
\end{picture}
\end{equation}
\vskip 3mm

Define the standard non-degenerate symmetric bilinear form $(\cdot,\cdot)$ on $\underline{\h}^*$ by
\begin{align}\label{bilinear:form}
(\ep_i,\ep_j)=\delta_{ij},\quad (\ep_i,\delta_k)=0,\quad (\delta_k,\delta_l)=-\delta_{kl},
\end{align}
where $1\le i,j\le m$ and $1\le k,l\le n$.

\subsection{Ortho-symplectic Lie superalgebras $\osp(\ell|2n)$}\label{sec:osp def}
Let $\G$ be the ortho-symplectic Lie superalgebra $\osp(\ell|2n)$ with $\ell=2m+1$ or $2m$ for $m,n\ge 1$. Fix the standard Cartan subalgebra $\h$ with its dual $\h^*$. We shall use the standard notation for the standard basis of $\h^*$ denoted by $\{\delta_1,\ldots,\delta_n,\ep_1,\ldots,\ep_m\}$ (see, e.g., \cite[Sections 1.2.4 and 1.2.5]{CW}) so that the standard Dynkin diagrams with simple roots indicated are as follows:

\begin{equation}\label{ABC:diagram:B1}
\hskip 1cm \setlength{\unitlength}{0.16in}
\begin{picture}(24,1)
\put(-.5,0.3){\makebox(0,0)[c]{$\osp(2m+1|2n)$}}
\put(5.7,0.5){\makebox(0,0)[c]{$\bigcirc$}}
\put(8,0.5){\makebox(0,0)[c]{$\bigcirc$}}
\put(10.4,0.5){\makebox(0,0)[c]{$\cdots$}}
\put(12.5,0.45){\makebox(0,0)[c]{$\bigotimes$}}
\put(14.85,0.5){\makebox(0,0)[c]{$\bigcirc$}}
\put(17.1,0.5){\makebox(0,0)[c]{$\cdots$}}
\put(19.3,0.5){\makebox(0,0)[c]{$\bigcirc$}}
\put(21.7,0.5){\makebox(0,0)[c]{$\bigcirc$}}
\put(6.1,0.5){\line(1,0){1.4}} \put(8.4,0.5){\line(1,0){1}}
\put(11,0.5){\line(1,0){1}} \put(13.1,0.5){\line(1,0){1.2}}
\put(15.28,0.5){\line(1,0){1}} \put(17.7,0.5){\line(1,0){1.1}}
\put(19.7,0.25){$\Longrightarrow$}
\put(5.5,-0.5){\makebox(0,0)[c]{\tiny $\delta_1-\delta_{2}$}}
\put(8.2,-0.5){\makebox(0,0)[c]{\tiny $\delta_{2}-\delta_{3}$}}
\put(12.2,-0.5){\makebox(0,0)[c]{\tiny $\delta_n-\ep_1$}}
\put(14.8,-0.5){\makebox(0,0)[c]{\tiny $\ep_1-\ep_2$}}
\put(19,-0.5){\makebox(0,0)[c]{\tiny $\ep_{m-1}-\ep_m$}}
\put(21.75,-0.5){\makebox(0,0)[c]{\tiny $\ep_{m}$}}
\end{picture}
\end{equation}

\begin{equation}\label{ABC:diagram:D1}
\hskip 1cm\setlength{\unitlength}{0.16in}
\begin{picture}(24,3)
\put(-1,0.3){\makebox(0,0)[c]{$\osp(2m|2n)$}}
\put(5.7,0.5){\makebox(0,0)[c]{$\bigcirc$}}
\put(8,0.5){\makebox(0,0)[c]{$\bigcirc$}}
\put(10.4,0.5){\makebox(0,0)[c]{$\cdots$}}
\put(12.5,0.45){\makebox(0,0)[c]{$\bigotimes$}}
\put(14.85,0.5){\makebox(0,0)[c]{$\bigcirc$}}
\put(17.45,0.5){\makebox(0,0)[c]{$\cdots$}}
\put(19.8,0.5){\makebox(0,0)[c]{$\bigcirc$}}
\put(21.6,2.2){\makebox(0,0)[c]{$\bigcirc$}}
\put(21.6,-1.2){\makebox(0,0)[c]{$\bigcirc$}}
\put(6.1,0.5){\line(1,0){1.4}}
\put(8.4,0.5){\line(1,0){1}}
\put(11,0.5){\line(1,0){1}}
\put(13.1,0.5){\line(1,0){1.2}}
\put(15.28,0.5){\line(1,0){1}}
\put(18.2,0.5){\line(1,0){1.1}}
\put(20.2,0.8){\line(1,1){1.1}}
\put(20.2,0.2){\line(1,-1){1.1}}
\put(5.5,-0.5){\makebox(0,0)[c]{\tiny $\delta_1-\delta_{2}$}}
\put(8.2,-0.5){\makebox(0,0)[c]{\tiny $\delta_{2}-\delta_{3}$}}
\put(12.3,-0.5){\makebox(0,0)[c]{\tiny $\delta_n-\ep_1$}}
\put(14.8,-0.5){\makebox(0,0)[c]{\tiny $\ep_1-\ep_2$}}
\put(18.8,-0.5){\makebox(0,0)[c]{\tiny $\ep_{m-2}-\ep_{m-1}$}}
\put(22.6,-2.2){\makebox(0,0)[c]{\tiny $\ep_{m-1}+\ep_{m}$}}
\put(22.6,3.2){\makebox(0,0)[c]{\tiny $\ep_{m-1}-\ep_{m}$}}
\end{picture}
\end{equation}
\vskip 1cm
\noindent Note that the basis above for $\h^*$ is precisely the basis for $\underline{\h}^*$ in Section \ref{sec:glnm} for the copies of $\gl(n|m)$ that are obtained by removing the last nodes corresponding to $\epsilon_m$ and $\epsilon_{m-1}+\epsilon_m$ in the Dynkin diagrams \eqref{ABC:diagram:B1} and \eqref{ABC:diagram:D1}, respectively.

We take our standard non-degenerate bilinear form $(\cdot,\cdot)$ on $\h^*$ to be the one determined by the same formulas as in \eqref{bilinear:form}.
We write $\mf{b}^{\texttt{st}}$ for the standard Borel subalgebras corresponding to \eqref{ABC:diagram:B1} and \eqref{ABC:diagram:D1}. Note that $\Phi^+_{{\mf b}^{\texttt{st}}}=\Phi^+_{{\mf b}^{\texttt{st}},\even}\cup \Phi^+_{{\mf b}^{\texttt{st}},\odd}$ is given by
\begin{equation}
\begin{cases}
\{\,\delta_i\pm\delta_j, 2\delta_p, \epsilon_k\pm\epsilon_l,\epsilon_q\,\}\cup\{\delta_p\pm \ep_q,\delta_p\,\}, & \text{for $\G=\osp(2m+1|2n)$},\\
\{\,\delta_i\pm\delta_j, 2\delta_p, \epsilon_k\pm \epsilon_l\,\}\cup\{\delta_p\pm \ep_q \,\}, & \text{for $\G=\osp(2m|2n)$},
\end{cases}
\end{equation}
where the indices are over $1\leq i<j\leq n$, $1\leq k<l\leq m$, $1\leq p\leq n$, and $1\leq q\leq m$, and $\rho^{{\mf b}^{\texttt{st}}}$ is given by
\begin{equation}\label{eq:rhost}
\begin{cases}
\sum_{i=1}^n\left(n-m-i+\hf\right)\delta_i + \sum_{j=1}^m(m-j+\hf)\ep_j, & \text{if $\G=\osp(2m+1|2n)$},\\
\sum_{i=1}^n(n-m-i+1)\delta_i + \sum_{j=1}^m(m-j)\ep_j, & \text{if $\G=\osp(2m|2n)$}.
\end{cases}
\end{equation}

We also consider the following Dynkin diagrams for $\G$:
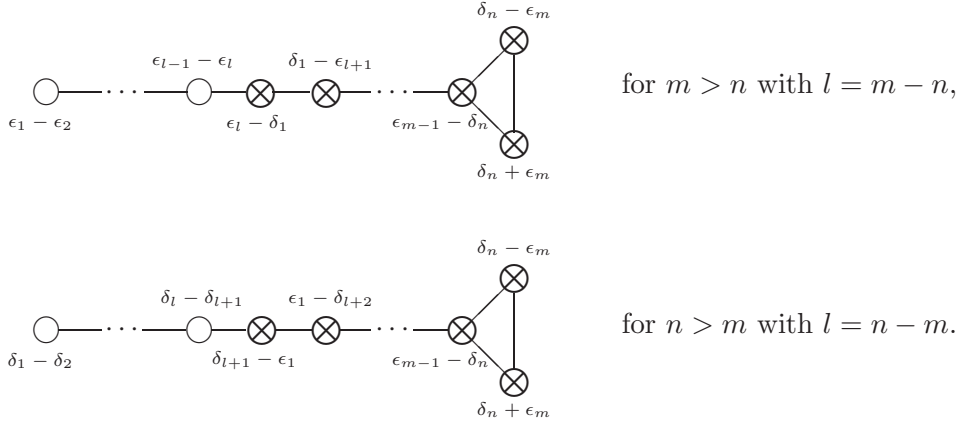
\begin{itemize}
\item[$\bullet$] $\G=\osp(2m+1|2n)$
\begin{center}
\begin{align*}\label{ABC:diagram:B2}
\hskip -4.2cm \setlength{\unitlength}{0.16in}
\begin{picture}(27,1)
\put(5.7,0.5){\makebox(0,0)[c]{$\bigotimes$}}
\put(8,0.5){\makebox(0,0)[c]{$\bigotimes$}}
\put(10.6,0.5){\makebox(0,0)[c]{$\cdots$}}
\put(12.6,0.45){\makebox(0,0)[c]{$\bigotimes$}}
\put(14.85,0.5){\makebox(0,0)[c]{$\bigotimes$}}
\put(17.2,0.5){\makebox(0,0)[c]{$\cdots$}}
\put(19.3,0.5){\makebox(0,0)[c]{$\bigotimes$}}
\put(21.7,0.5){\makebox(0,0)[c]{\CircleSolid}}
\put(6.1,0.5){\line(1,0){1.4}} \put(8.5,0.5){\line(1,0){1.4}}
\put(11.2,0.5){\line(1,0){1}} \put(13.1,0.5){\line(1,0){1.2}}
\put(15.28,0.5){\line(1,0){1}} \put(17.7,0.5){\line(1,0){1.1}}
\put(19.7,0.25){$\Longrightarrow$}
\put(5.5,-0.5){\makebox(0,0)[c]{\tiny $\ep_1-\delta_1$}}
\put(8,-0.5){\makebox(0,0)[c]{\tiny $\delta_1-\ep_{2}$}}
\put(19.2,-0.5){\makebox(0,0)[c]{\tiny $\ep_{m}-\delta_m$}}
\put(21.75,-0.5){\makebox(0,0)[c]{\tiny $\delta_{m}$}}
\put(26.5,0.7){\makebox(0,0)[c]{for $m=n$,}}
\end{picture}
\end{align*}
\vskip 0.4cm
\end{center}
 \begin{center}
\begin{equation*}\label{ABC:diagram:B3}
\hskip -3cm \setlength{\unitlength}{0.16in}
\begin{picture}(30,1.5)
\put(5.7,0.5){\makebox(0,0)[c]{$\bigcirc$}}
\put(8.25,0.5){\makebox(0,0)[c]{$\cdots$}}
\put(10.7,0.5){\makebox(0,0)[c]{$\bigcirc$}}
\put(12.7,0.45){\makebox(0,0)[c]{$\bigotimes$}}
\put(14.85,0.5){\makebox(0,0)[c]{$\bigotimes$}}
\put(17.1,0.5){\makebox(0,0)[c]{$\cdots$}}
\put(19.3,0.5){\makebox(0,0)[c]{$\bigotimes$}}
\put(21.7,0.5){\makebox(0,0)[c]{\CircleSolid}}
\put(6.1,0.5){\line(1,0){1.4}} \put(8.9,0.5){\line(1,0){1.4}}
\put(11.1,0.5){\line(1,0){1.1}} \put(13.1,0.5){\line(1,0){1.2}}
\put(15.28,0.5){\line(1,0){1}} \put(17.7,0.5){\line(1,0){1.1}}
\put(19.7,0.25){$\Longrightarrow$}
\put(5.5,-0.5){\makebox(0,0)[c]{\tiny $\ep_1-\ep_{2}$}}
\put(10.5,1.5){\makebox(0,0)[c]{\tiny $\ep_{l}-\ep_{l+1}$}}
\put(12.4,-0.5){\makebox(0,0)[c]{\tiny $\ep_{l+1}-\delta_1$}}
\put(15,1.5){\makebox(0,0)[c]{\tiny $\delta_1-\ep_{l+2}$}}
\put(19,-0.5){\makebox(0,0)[c]{\tiny $\ep_{m}-\delta_n$}}
\put(21.75,-0.5){\makebox(0,0)[c]{\tiny $\delta_{n}$}}
\put(30,0.7){\makebox(0,0)[c]{for $m>n$ with $l=m-n,$}}
\end{picture}
\end{equation*}
\vskip 0.4cm
\end{center}
\begin{center}
\begin{equation*}\label{ABC:diagram:B4}
\hskip -3cm \setlength{\unitlength}{0.16in}
\begin{picture}(30,1.5)
\put(5.7,0.5){\makebox(0,0)[c]{$\bigcirc$}}
\put(8.25,0.5){\makebox(0,0)[c]{$\cdots$}}
\put(10.7,0.5){\makebox(0,0)[c]{$\bigcirc$}}
\put(12.7,0.5){\makebox(0,0)[c]{$\bigotimes$}}
\put(14.85,0.5){\makebox(0,0)[c]{$\bigotimes$}}
\put(17.1,0.5){\makebox(0,0)[c]{$\cdots$}}
\put(19.3,0.5){\makebox(0,0)[c]{$\bigotimes$}}
\put(21.7,0.5){\makebox(0,0)[c]{\CircleSolid}}
\put(6.1,0.5){\line(1,0){1.4}} \put(8.85,0.5){\line(1,0){1.4}}
\put(11.1,0.5){\line(1,0){1.1}} \put(13.1,0.5){\line(1,0){1.2}}
\put(15.28,0.5){\line(1,0){1}} \put(17.7,0.5){\line(1,0){1.1}}
\put(19.7,0.25){$\Longrightarrow$}
\put(5.5,-0.5){\makebox(0,0)[c]{\tiny $\delta_1-\delta_{2}$}}
\put(10.5,1.5){\makebox(0,0)[c]{\tiny $\delta_{l-1}-\delta_{l}$}}
\put(12.5,-0.5){\makebox(0,0)[c]{\tiny $\delta_{l}-\ep_1$}}
\put(15,1.5){\makebox(0,0)[c]{\tiny $\ep_1 -\delta_{l+1}$}}
\put(19,-0.5){\makebox(0,0)[c]{\tiny $\ep_{m}-\delta_n$}}
\put(21.75,-0.5){\makebox(0,0)[c]{\tiny $\delta_{n}$}}
\put(30,0.7){\makebox(0,0)[c]{for $n>m$ with $l=n-m.$}}
\end{picture}
\end{equation*}
\vskip 0.8cm
\end{center}

\item[$\bullet$] $\G=\osp(2m|2n)$
\begin{center}
\begin{align*}\label{ABC:diagram:D2}
\hskip -4.2cm \setlength{\unitlength}{0.16in}
\begin{picture}(27,3)
\put(5.7,0.5){\makebox(0,0)[c]{$\bigotimes$}}
\put(8,0.5){\makebox(0,0)[c]{$\bigotimes$}}
\put(10.6,0.5){\makebox(0,0)[c]{$\cdots$}}
\put(12.6,0.45){\makebox(0,0)[c]{$\bigotimes$}}
\put(14.85,0.5){\makebox(0,0)[c]{$\bigotimes$}}
\put(17.2,0.5){\makebox(0,0)[c]{$\cdots$}}
\put(19.3,0.5){\makebox(0,0)[c]{$\bigotimes$}}
\put(21,2.2){\makebox(0,0)[c]{$\bigotimes$}}
\put(21,-1.2){\makebox(0,0)[c]{$\bigotimes$}}
\put(6.1,0.5){\line(1,0){1.4}} \put(8.5,0.5){\line(1,0){1.4}}
\put(11.2,0.5){\line(1,0){1}} \put(13.1,0.5){\line(1,0){1.2}}
\put(15.28,0.5){\line(1,0){1}} \put(17.7,0.5){\line(1,0){1.1}}
\put(19.6,0.8){\line(1,1){1.1}}
\put(19.6,0.2){\line(1,-1){1.1}}
\put(21,1.7){\line(0,-1){2.5}}
\put(5.5,-0.5){\makebox(0,0)[c]{\tiny $\delta_1-\ep_1$}}
\put(8,-0.5){\makebox(0,0)[c]{\tiny $\ep_1-\delta_{2}$}}
\put(18.6,-0.5){\makebox(0,0)[c]{\tiny $\ep_{m-1}-\delta_m$}}
\put(21,-2.1){\makebox(0,0)[c]{\tiny $\delta_{m}+\ep_m$}}
\put(21,3.2){\makebox(0,0)[c]{\tiny $\delta_{m}-\ep_m$}}
\put(26.5,0.7){\makebox(0,0)[c]{for $m=n$,}}
\end{picture}
\end{align*}
\vskip 1.4cm
\end{center}
 \begin{center}
\begin{equation*}\label{ABC:diagram:D3}
\hskip -3cm \setlength{\unitlength}{0.16in}
\begin{picture}(30,2)
\put(5.7,0.5){\makebox(0,0)[c]{$\bigcirc$}}
\put(8.25,0.5){\makebox(0,0)[c]{$\cdots$}}
\put(10.7,0.5){\makebox(0,0)[c]{$\bigcirc$}}
\put(12.7,0.45){\makebox(0,0)[c]{$\bigotimes$}}
\put(14.85,0.45){\makebox(0,0)[c]{$\bigotimes$}}
\put(17.1,0.5){\makebox(0,0)[c]{$\cdots$}}
\put(19.3,0.5){\makebox(0,0)[c]{$\bigotimes$}}
\put(21,2.2){\makebox(0,0)[c]{$\bigotimes$}}
\put(21,-1.2){\makebox(0,0)[c]{$\bigotimes$}}
\put(6.1,0.5){\line(1,0){1.4}} \put(8.9,0.5){\line(1,0){1.4}}
\put(11.1,0.5){\line(1,0){1.1}} \put(13.1,0.5){\line(1,0){1.2}}
\put(15.28,0.5){\line(1,0){1}} \put(17.7,0.5){\line(1,0){1.1}}
\put(19.6,0.8){\line(1,1){1.1}}
\put(19.6,0.2){\line(1,-1){1.1}}
\put(21,1.7){\line(0,-1){2.5}}
\put(5.5,-0.5){\makebox(0,0)[c]{\tiny $\ep_1-\ep_{2}$}}
\put(10.5,1.5){\makebox(0,0)[c]{\tiny $\ep_{l-1}-\ep_{l}$}}
\put(12.6,-0.5){\makebox(0,0)[c]{\tiny $\ep_{l}-\delta_1$}}
\put(15,1.5){\makebox(0,0)[c]{\tiny $\delta_1-\ep_{l+1}$}}
\put(18.6,-0.5){\makebox(0,0)[c]{\tiny $\ep_{m-1}-\delta_n$}}
\put(21,-2.1){\makebox(0,0)[c]{\tiny $\delta_{n}+\ep_m$}}
\put(21,3.2){\makebox(0,0)[c]{\tiny $\delta_{n}-\ep_m$}}
\put(30,0.7){\makebox(0,0)[c]{for $m>n$ with $l=m-n$,}}
\end{picture}
\end{equation*}
\vskip 1.4cm
\end{center}
\begin{center}
\begin{equation*}\label{ABC:diagram:D4}
\hskip -3cm \setlength{\unitlength}{0.16in}
\begin{picture}(30,2)
\put(5.7,0.5){\makebox(0,0)[c]{$\bigcirc$}}
\put(8.25,0.5){\makebox(0,0)[c]{$\cdots$}}
\put(10.7,0.5){\makebox(0,0)[c]{$\bigcirc$}}
\put(12.75,0.5){\makebox(0,0)[c]{$\bigotimes$}}
\put(14.85,0.5){\makebox(0,0)[c]{$\bigotimes$}}
\put(17.1,0.5){\makebox(0,0)[c]{$\cdots$}}
\put(19.3,0.5){\makebox(0,0)[c]{$\bigotimes$}}
\put(21,2.2){\makebox(0,0)[c]{$\bigotimes$}}
\put(21,-1.2){\makebox(0,0)[c]{$\bigotimes$}}
\put(6.1,0.5){\line(1,0){1.4}} \put(8.85,0.5){\line(1,0){1.4}}
\put(11.1,0.5){\line(1,0){1.1}} \put(13.2,0.5){\line(1,0){1.2}}
\put(15.28,0.5){\line(1,0){1}} \put(17.7,0.5){\line(1,0){1.1}}
\put(19.6,0.8){\line(1,1){1.1}}
\put(19.6,0.2){\line(1,-1){1.1}}
\put(21,1.7){\line(0,-1){2.5}}
\put(5.5,-0.5){\makebox(0,0)[c]{\tiny $\delta_1-\delta_{2}$}}
\put(10.8,1.5){\makebox(0,0)[c]{\tiny $\delta_{l}-\delta_{l+1}$}}
\put(12.5,-0.5){\makebox(0,0)[c]{\tiny $\delta_{l+1}-\ep_1$}}
\put(15,1.5){\makebox(0,0)[c]{\tiny $\ep_1 -\delta_{l+2}$}}
\put(18.6,-0.5){\makebox(0,0)[c]{\tiny $\ep_{m-1}-\delta_n$}}
\put(21,-2.1){\makebox(0,0)[c]{\tiny $\delta_{n}+\ep_m$}}
\put(21,3.2){\makebox(0,0)[c]{\tiny $\delta_{n}-\ep_m$}}
\put(30,0.7){\makebox(0,0)[c]{for $n>m$ with $l=n-m$.}}
\end{picture}
\end{equation*}
\vskip 1cm
\end{center}

\end{itemize}

We shall denote by $\mf b^{\texttt{odd}}$ the Borel subalgebras of $\G$ corresponding to the Dynkin diagrams above. As the notation indicates, these Borel subalgebras contain the maximal number of odd isotropic simple roots. They play a fundamental role in \cite{GS} and they shall play an important role in this article as well.

Let $V$ be a finite-dimensional irreducible $\G$-module. Recall the decomposition into weight spaces $V=\bigoplus_{\gamma\in\h^*}V_\gamma$. We say that $V$ is an {\em integer weight module}, if for all $\gamma$ with $V_\gamma\not=0$, we have $\gamma=\sum_{i=1}^n\gamma_i\delta_i+\sum_{j=1}^m\kappa_j\ep_j$ with $\kappa_j\in\Z$. The module $V$ is called a {\em half-integer weight module}, if for all $\gamma$ with $V_\gamma\not=0$, we have $\gamma=\sum_{i=1}^n\gamma_i\delta_i+\sum_{j=1}^m\kappa_j\ep_j$ with $\kappa_j\in\hf+\Z$. Note that if $V$ is a finite-dimensional irreducible half-integer weight module, then $V$ (with respect to any Borel subalgebra) is typical. Thus, in the sequel we shall restrict ourselves to the interesting case of integer weight modules.

Let ${\mf b}$ be a Borel subalgebra.
Given an odd isotropic simple root $\alpha\in \Pi_{\mf b}$, let $s_\alpha$ be the corresponding odd reflection.
Let ${\mf b}'$ be the Borel subalgebra with fundamental system $\Pi_{\mf b'}=s_\alpha(\Pi_{\mf b})$.
Consider a finite-dimensional irreducible $\G$-module $V=L_{\mf b}(\G,\gamma)$ with ${\mf b}$-highest weight $\gamma$.  If $\gamma'$ denotes the ${\mf b}'$-highest weight of $V$, that is, $L_{\mf b}(\G,\gamma)\cong L_{\mf b'}(\G,\gamma')$, then \cite[Lemma 1]{PS} (see also \cite[Lemma 1.40]{CW})
\begin{equation}\label{eq:odd reflection rho}
\gamma'+\rho^{{\mf b}'}=
\begin{cases}
\gamma + \rho^{{\mf b}}, & \text{if $(\gamma,\alpha)\neq 0$},\\
\gamma + \rho^{{\mf b}} + \alpha , & \text{if $(\gamma,\alpha)= 0$}.
\end{cases}
\end{equation}

\subsection{Borel subalgebras and polynomial $\gl(n|m)$-modules}\label{subsec:gl:po}

Recall from Section \ref{sec:glnm} the ordered basis $\{v_1,\ldots,v_n,w_1,\ldots,w_m\}$ of $\C^{n|m}$ so that the Lie superalgebra $\mf l=\gl(n|m)$ can be realized as $(n+m)\times (n+m)$ complex matrices. A total ordering of the basis $\{v_1,\ldots,v_n,w_1,\ldots,w_m\}$ that preserves the ordering among the $v_i$'s and the $w_j$'s gives rise to a Borel subalgebra $\underline{\mf b}$ of $\mf l$ with $\underline{\mf b}_\even=\left(\underline{\mf b}^{\tt st}\right)_\even$.  Such an ordering $\cdots v_i\cdots w_j\cdots$ is clearly determined by replacing all $v_i$'s with $\delta$'s and all $w_j$'s with $\ep$'s so that we get a sequence with $m$ $\ep$'s and $n$ $\delta$'s, which we call an {\it $\ep\delta$-sequence}. It is known that the $W_{\mf l}$-conjugacy classes of the Borel subalgebras of $\mf l$, or the Borel subalgebras $\underline{\mf b}$ of ${\mf l}$ with $\underline{\mf b}_\even=\left(\underline{\mf b}^{\tt st}\right)_\even$,  are in one-to-one correspondence with such total orderings and hence with such $\ep\delta$-sequences (see e.g.~\cite{K2} or \cite[Section 1.3.2]{CW}). In fact, we can produce a Dynkin diagram for $\mf{l}$ corresponding to such an $\ep\delta$-sequence as follows: We number all the $m$ $\ep$'s in the $\ep\delta$-sequence from left to right starting with $\ep_1$ and ending with $\ep_m$. Similarly, we label the $n$ $\delta$'s in the sequence from left to right starting with $\delta_1$ and ending with $\delta_n$. We call this $\ep\delta$-sequence a {\em numbered $\ep\delta$-sequence}. Now, from left to right we write all the differences of two consecutive members in the numbered $\ep\delta$-sequence. This way we get a fundamental system for $\mf l$. The $W_{\mf l}$-conjugacy class of this fundamental system then corresponds uniquely to this $\ep\delta$-sequence.

\begin{example}\label{ex:00}
Consider $\gl(2|3)$ with $n=2$ and $m=3$.  Take $\ep\delta$-sequences $\ep\ep\delta\ep\delta$ and $\ep\ep\delta\delta\ep$. The numbered sequences are $\ep_1\ep_2\delta_1\ep_3\delta_2$ and $\ep_1\ep_2\delta_1\delta_2\ep_3$, respectively. By taking all consecutive differences we get the following sets of fundamental systems:
\begin{align*}
&\{\ep_1-\ep_2, \ep_2-\delta_1, \delta_1-\ep_3,\ep_3-\delta_2\},\\
&\{\ep_1-\ep_2, \ep_2-\delta_1, \delta_1-\delta_2,\delta_2-\ep_3\},
\end{align*}
with respective Dynkin diagrams:
\begin{center}
\begin{equation}\label{ABC:diagram:A2}
\hskip -2cm
\setlength{\unitlength}{0.2in}
\begin{picture}(24,1)
\put(5.7,0.5){\makebox(0,0)[c]{$\bigcirc$}}
\put(8,0.5){\makebox(0,0)[c]{$\bigotimes$}}
\put(10,0.5){\makebox(0,0)[c]{$\bigotimes$}}
\put(12,0.45){\makebox(0,0)[c]{$\bigotimes$}}
\put(6.1,0.5){\line(1,0){1.5}} \put(8.3,0.5){\line(1,0){1.4}}
\put(10.3,0.5){\line(1,0){1.4}}
\put(5.5,-0.5){\makebox(0,0)[c]{\tiny $\ep_1-\ep_{2}$}}
\put(8,-0.5){\makebox(0,0)[c]{\tiny $\ep_{2}-\delta_{1}$}}
\put(10,-0.5){\makebox(0,0)[c]{\tiny $\delta_{1}-\ep_{3}$}}
\put(12.1,-0.5){\makebox(0,0)[c]{\tiny $\ep_3-\delta_2$}}
\put(15.7,0.5){\makebox(0,0)[c]{$\bigcirc$}}
\put(18,0.5){\makebox(0,0)[c]{$\bigotimes$}}
\put(20,0.5){\makebox(0,0)[c]{$\bigcirc$}}
\put(22,0.45){\makebox(0,0)[c]{$\bigotimes$}}
\put(16.1,0.5){\line(1,0){1.5}} \put(18.3,0.5){\line(1,0){1.4}}
\put(20.3,0.5){\line(1,0){1.4}}
\put(15.5,-0.5){\makebox(0,0)[c]{\tiny $\ep_1-\ep_{2}$}}
\put(18,-0.5){\makebox(0,0)[c]{\tiny $\ep_{2}-\delta_{1}$}}
\put(20,-0.5){\makebox(0,0)[c]{\tiny $\delta_{1}-\delta_{2}$}}
\put(22.1,-0.5){\makebox(0,0)[c]{\tiny $\delta_2-\ep_3$}}
\end{picture}
\end{equation}
\vskip 0.4cm
\end{center}
\end{example}
The Dynkin diagram attached to the standard Borel subalgebra $\underline{\mf{b}}^{\texttt{st}}$ in \eqref{ABC:diagram:A1} corresponds to the $\ep\delta$-sequence
$\underbrace{\delta\ldots\delta}_n\underbrace{\ep\ldots\ep}_m$.
%$\delta\ldots\delta  {\ep\ldots\ep}$, i.e., $n$ $\delta$'s followed by $m$ $\ep$'s.

Recall that a partition $\la=(\la_1,\la_2,\ldots)$ is called an $(n|m)$-hook partition, if it satisfies $\la_{n+1}\le m$. We denote the set of $(n|m)$-hook partitions by $\mc{H}(n|m)$. Given $\la\in\mc{H}(n|m)$, we define weights $\la^\natural, \la_-^\natural\in\underline{\h}^*$ by
\begin{equation}\label{lambda:natural}
\begin{split}
\la^\natural&:=\sum_{i=1}^n\la_i\delta_i+\sum_{j=1}^m\kappa_j\ep_j,\\
\la_-^\natural&:=\sum_{i=1}^n\la_i\delta_i+\sum_{j=1}^{m-1}\kappa_j\ep_j-\kappa_m\ep_m.
\end{split}
\end{equation}
where $(\kappa_1,\ldots,\kappa_m)$ is the transpose of the partition $(\la_{n+1},\la_{n+2},\ldots)$.

It is known that the finite-dimensional irreducible polynomial $\mf l$-modules are parameterized by $\mc{H}(n|m)$, that is, an irreducible  $\mf l$-module with $\underline{\mf b}^{\texttt{st}}$-highest weight $\gamma$ is an irreducible polynomial ${\mf l}$-module if and only if $\gamma=\la^\natural$ for some $\la\in\mc{H}(n|m)$ \cite{Sv,BR} (see also \cite[Section 3.2.6]{CW}).

Now, take a Borel subalgebra $\underline{\mf{b}}$ with $\underline{\mf{b}}_{\bar 0}=(\underline{\mf{b}}^{\texttt{st}})_{\bar 0}$. Then $\underline{\mf b}$ corresponds to an $\ep\delta$-sequence of the form
$\delta^{d_1} \ep^{e_1}\delta^{d_2}\ep^{e_2}\cdots\delta^{d_r}\ep^{e_r}$, where the exponents denote the corresponding
multiplicities
(all $d_i,e_i$ are positive except possibly $d_1 =0$ or $e_r =0$).
Let
\begin{equation*}\label{the:d:e}
\texttt d_u :=\sum_{a=1}^u d_a, \quad \quad \texttt e_u
:=\sum_{a=1}^u e_a,
\end{equation*}
for $1\leq u\leq  r$. Set $\texttt d_0 =\texttt e_0 =0$. Note that
$\texttt d_r=n$ and $\texttt e_r=m$. For each $\la\in \mc{H}(n|m)$, we define the {\em Frobenius
coordinates} $(p_i|q_j)$ of $\la$ corresponding to $\underline{\mf{b}}$ as
follows \cite[Section 2.4]{CW}:
\begin{align*}
\begin{split}
\begin{cases}
p_i = \max \{\la_i - \texttt e_u, 0\},  & \text{ if }\texttt d_u
< i \leq \texttt d_{u+1}  \text{ for some } 0\le u\le
r-1,\label{block:frobenius}
  \\
q_j  = \max \{\la_j' - \texttt d_{u+1}, 0\},&  \text{ if }
\texttt e_u < j \leq \texttt e_{u+1}  \text{ for some } 0\le u\le
r-1,
\end{cases}
\end{split}
\end{align*}
for $1\le i \le n$ and $1\le j \le m$. Here $\la'=(\la'_1,\la'_2,\ldots)$ denotes the transpose of $\la$. Then we define weights $\la^{\underline{\mf{b}}}, \la_-^{\underline{\mf{b}}} \in \underline{\h}^*$ by
\begin{equation}\label{frob:ht:wt}
\begin{split}
\la^{\underline{\mf{b}}}& :=\sum_{i=1}^n  p_i \delta_{i} +\sum_{j=1}^m  q_j
\epsilon_{j},\\
\la_-^{\underline{\mf{b}}}& :=\sum_{i=1}^n  p_i \delta_{i} +\sum_{j=1}^{m-1}  q_j
\epsilon_{j}-q_m\ep_m.
\end{split}
\end{equation}
Note that  $\la^\natural=\la^{\underline{\mf b}}$ when $\underline{\mf b} =\underline{\mf b}^{\texttt{st}}$, and we have \cite{CLW} (see also \cite[Theorem 2.55]{CW})
\begin{align*}
L_{\underline{\mf{b}}^{\texttt{st}}}(\mf l,\la^\natural)=L_{\underline{\mf{b}}^{\texttt{st}}}(\mf l,\la^{\underline{\mf{b}}^{\texttt{st}}}) \cong L_{\underline{\mf{b}}}(\mf l,\la^{\underline{\mf{b}}}).
\end{align*}

\subsection{Borel subalgebras and finite-dimensional $\osp(2m+1|2n)$-modules}\label{sec:Borels B}

Suppose that $\G=\osp(2m+1|2n)$.  We assume that ${\mf l}=\gl(n|m)$ is the subalgebra of $\G$ corresponding to the subdiagram \eqref{ABC:diagram:A1} of \eqref{ABC:diagram:B1}.

According to \cite{K2} (see also \cite[Section 1.3.3]{CW}) the $W$-conjugacy classes of Borel subalgebras ${\mf b}$ of $\G$, or the Borel subalgebras ${\mf b}$ of $\G$ with ${\mf b}_\even=\left({\mf b}^{\tt st}\right)_\even$, are in one-to-one correspondence with $\ep\delta$-sequences with $m$ $\ep$'s and $n$ $\delta$'s.  In fact, we can produce a fundamental system $\Pi$ and its Dynkin diagram corresponding to such an $\ep\delta$-sequence as follows:  We first construct the  fundamental system $\underline{\Pi}$ for ${\mf l}$ associated to the $\epsilon\delta$-sequence as in Section \ref{subsec:gl:po} so that we get a corresponding Dynkin diagram $\texttt{D}$. Next, if the $\ep\delta$-sequence ends with an $\ep$, then we put $\Pi=\underline{\Pi}\cup \{\ep_m\}$ and attach $\bigcirc$ for the simple short root $\ep_m$ to the right end of $\texttt D$, and if the sequence ends with a $\delta$, then we put $\Pi=\underline{\Pi}\cup \{\delta_n\}$ and attach \raisebox{-0.5ex}{\CircleSolid} for the non-isotropic odd simple root $\delta_n$ to the right end of $\texttt D$.

Note that the sequence $\underbrace{\delta\ldots\delta}_n\underbrace{\ep\ldots\ep}_m$ gives rise to the diagram \eqref{ABC:diagram:B1} and hence corresponds to ${\mf b}^{\tt st}$, while the sequence $\underbrace{\ep\ldots\ep}_m\underbrace{\delta\ldots\delta}_n$ \ gives the diagram \vskip 3mm
\begin{center}
\vskip -0.3cm
\begin{equation*}\label{ABC:diagram:B5}
\hskip -1cm \setlength{\unitlength}{0.16in}
\begin{picture}(24,1)
\put(5.7,0.5){\makebox(0,0)[c]{$\bigcirc$}}
\put(8,0.5){\makebox(0,0)[c]{$\bigcirc$}}
\put(10.4,0.5){\makebox(0,0)[c]{$\cdots$}}
\put(12.5,0.45){\makebox(0,0)[c]{$\bigotimes$}}
\put(14.85,0.5){\makebox(0,0)[c]{$\bigcirc$}}
\put(17.25,0.5){\makebox(0,0)[c]{$\cdots$}}
\put(19.3,0.5){\makebox(0,0)[c]{$\bigcirc$}}
\put(21.7,0.5){\makebox(0,0)[c]{\CircleSolid}}
\put(6.1,0.5){\line(1,0){1.4}} \put(8.4,0.5){\line(1,0){1}}
\put(11,0.5){\line(1,0){1}} \put(13.1,0.5){\line(1,0){1.2}}
\put(15.28,0.5){\line(1,0){1}} \put(17.7,0.5){\line(1,0){1.1}}
\put(19.7,0.25){$\Longrightarrow$}
\put(5.2,-0.5){\makebox(0,0)[c]{\tiny $\ep_1-\ep_{2}$}}
\put(8.2,-0.5){\makebox(0,0)[c]{\tiny $\ep_{2}-\ep_{3}$}}
\put(12.2,-0.5){\makebox(0,0)[c]{\tiny $\ep_m-\delta_1$}}
\put(14.8,-0.5){\makebox(0,0)[c]{\tiny $\delta_1-\delta_2$}}
\put(19,-0.5){\makebox(0,0)[c]{\tiny $\delta_{n-1}-\delta_n$}}
\put(21.75,-0.5){\makebox(0,0)[c]{\tiny $\delta_{n}$}}
\end{picture}
\end{equation*}
\vskip 0.5cm
\end{center}

\begin{example}
Consider $\osp(7|4)$ with sequences $\ep\ep\delta\ep\delta$ and $\ep\ep\delta\delta\ep$ as in Example \ref{ex:00} so that we get the diagrams in \eqref{ABC:diagram:A2} with respective fundamental systems
\begin{align*}
&\{\ep_1-\ep_2, \ep_2-\delta_1, \delta_1-\ep_3,\ep_3-\delta_2\},\\
&\{\ep_1-\ep_2, \ep_2-\delta_1, \delta_1-\delta_2,\delta_2-\ep_3\}.
\end{align*}
The respective Dynkin diagrams for $\osp(7|4)$ are then obtained by attaching simple roots $\delta_2$ and $\ep_3$, respectively, to the right-most nodes, and they are
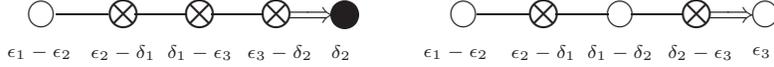
\begin{center}
\begin{equation*}\label{ABC:diagram:BB1}
\hskip -2cm
\setlength{\unitlength}{0.2in}
\begin{picture}(24,1)
\put(4.85,0.5){\makebox(0,0)[c]{$\bigcirc$}}
\put(7,0.5){\makebox(0,0)[c]{$\bigotimes$}}
\put(9,0.5){\makebox(0,0)[c]{$\bigotimes$}}
\put(11,0.5){\makebox(0,0)[c]{$\bigotimes$}}
\put(12.8,0.5){\makebox(0,0)[c]{\CircleSolid}}
\put(5.25,0.5){\line(1,0){1.35}} \put(7.3,0.5){\line(1,0){1.35}}
\put(9.3,0.5){\line(1,0){1.35}}
\put(11.3,0.3){$\Longrightarrow$}
\put(4.8,-0.5){\makebox(0,0)[c]{\tiny $\ep_1-\ep_{2}$}}
\put(7,-0.5){\makebox(0,0)[c]{\tiny $\ep_{2}-\delta_{1}$}}
\put(9,-0.5){\makebox(0,0)[c]{\tiny $\delta_{1}-\ep_{3}$}}
\put(11.1,-0.5){\makebox(0,0)[c]{\tiny $\ep_3-\delta_2$}}
\put(12.7,-0.5){\makebox(0,0)[c]{\tiny $\delta_2$}}
\put(15.9,0.5){\makebox(0,0)[c]{$\bigcirc$}}
\put(18,0.5){\makebox(0,0)[c]{$\bigotimes$}}
\put(20,0.5){\makebox(0,0)[c]{$\bigcirc$}}
\put(22,0.5){\makebox(0,0)[c]{$\bigotimes$}}
\put(23.8,0.5){\makebox(0,0)[c]{$\bigcirc$}}
\put(16.25,0.5){\line(1,0){1.35}} \put(18.3,0.5){\line(1,0){1.4}}
\put(20.3,0.5){\line(1,0){1.35}}
\put(22.3,0.3){$\Longrightarrow$}
\put(15.7,-0.5){\makebox(0,0)[c]{\tiny $\ep_1-\ep_{2}$}}
\put(18,-0.5){\makebox(0,0)[c]{\tiny $\ep_{2}-\delta_{1}$}}
\put(20,-0.5){\makebox(0,0)[c]{\tiny $\delta_{1}-\delta_{2}$}}
\put(22.1,-0.5){\makebox(0,0)[c]{\tiny $\delta_2-\ep_3$}}
\put(23.7,-0.5){\makebox(0,0)[c]{\tiny $\ep_3$}}
\end{picture}
\end{equation*}
\vskip 0.4cm
\end{center}
\end{example}

It is known that the finite-dimensional irreducible integer weight $\G$-modules are parameterized by $\mc{H}(n|m)$ \cite{K1} (see also \cite[Theorem 2.11]{CW}). Note that given $\la\in\mc{H}(n|m)$, the weight $\la^\natural$ in \eqref{lambda:natural} can be regarded as a weight in $\h^*$ for $\G$.
{Then given $\gamma\in \h^*$, $\gamma$ is a ${\mf b}^{\tt st}$-highest weight of a finite-dimensional irreducible integer weight $\G$-module $V$ if and only if $\gamma=\la^\natural$ for some $\la\in\mc{H}(n|m)$,} that is, $V\cong L_{\mf b^{\texttt{st}}}(\G,\la^\natural)= L(\la^\natural)$.

Now let $\mf b$ be a Borel subalgebra of $\G$ with $\mf b_{\bar 0}=\left(\mf b^{\texttt{st}}\right)_{\bar 0}$.
Then $\mf b$ or $\underline{\mf b}={\mf b}\cap {\mf l}$ determines a unique $\ep\delta$-sequence.
For $\la\in\mc{H}(n|m)$, let $\la^{{\mf b}}$ denote the ${\mf b}$-highest weight for $L(\la^\natural)$, that is, $L_{\mf b}(\G,\la^{\mf b})\cong  L(\la^\natural)$. By
\cite{CLW}, \cite[Theorem 2.58]{CW}, we have
\begin{align}\label{eq:aux004}
\la^{{\mf b}}=\la^{\underline{\mf b}},
\end{align}
where $\la^{\underline{\mf b}}$ is given in \eqref{frob:ht:wt}.

\begin{example}\label{frob:block}
Consider the $(5|4)$-hook partition $\la=(10,9,6,4,4,4,3,2,1,1,1)$.
Let $\mf b$ be the Borel subalgebra of $\osp(9|10)$ associated to the
following fundamental system:

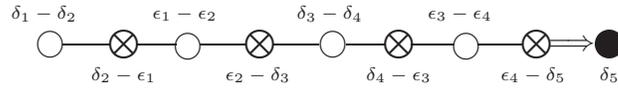
\begin{center}
\hskip -3cm \setlength{\unitlength}{0.16in}
\begin{picture}(24,4)
\put(6,2){\makebox(0,0)[c]{$\bigcirc$}}
\put(8.4,2){\makebox(0,0)[c]{$\bigotimes$}}
\put(10.5,1.95){\makebox(0,0)[c]{$\bigcirc$}}
\put(12.85,2){\makebox(0,0)[c]{$\bigotimes$}}
\put(15.25,2){\makebox(0,0)[c]{$\bigcirc$}}
\put(17.4,2){\makebox(0,0)[c]{$\bigotimes$}}
\put(19.6,1.95){\makebox(0,0)[c]{$\bigcirc$}}
\put(21.9,2){\makebox(0,0)[c]{$\bigotimes$}}
\put(24.3,2){\makebox(0,0)[c]{\CircleSolid}}
\put(6.4,2){\line(1,0){1.55}} \put(8.82,2){\line(1,0){1.3}}
\put(10.9,2){\line(1,0){1.5}} \put(13.28,2){\line(1,0){1.5}}
\put(15.7,2){\line(1,0){1.25}} \put(17.8,2){\line(1,0){1.4}}
\put(20,2){\line(1,0){1.4}}
\put(22.3,1.8){$\Longrightarrow$}
\put(5.8,3){\makebox(0,0)[c]{\tiny $\delta_1-\delta_2$}}
\put(8.4,1){\makebox(0,0)[c]{\tiny $\delta_2-\epsilon_1$}}
\put(10.4,3){\makebox(0,0)[c]{\tiny $\epsilon_1-\epsilon_2$}}
\put(12.8,1){\makebox(0,0)[c]{\tiny $\epsilon_2-\delta_3$}}
\put(15.15,3){\makebox(0,0)[c]{\tiny $\delta_3-\delta_4$}}
\put(17.4,1){\makebox(0,0)[c]{\tiny $\delta_4-\epsilon_3$}}
\put(19.4,3){\makebox(0,0)[c]{\tiny $\epsilon_3-\epsilon_4$}}
\put(21.8,1){\makebox(0,0)[c]{\tiny $\epsilon_4-\delta_5$}}
\put(24.3,1){\makebox(0,0)[c]{\tiny $\delta_{5}$}}
\end{picture}
\end{center}
Then we have $\la^{\mf b}=10\delta_1+9\delta_2+4\delta_3+2\delta_4+9\ep_1+6\ep_2+3\ep_3+2\ep_4$.
\end{example}

\subsection{Borel subalgebras and finite-dimensional $\osp(2m|2n)$-modules}\label{sec:Borels D}

Suppose that $\G=\osp(2m|2n)$. We note that when $m=1$ the Lie superalgebra $\osp(2|2n)$ is of type I, and its finite-dimensional representation theory is different from the case when $m\ge 2$. It is also easier, since all finite-dimensional irreducible modules have degree of atypicality at most one. In this case the atypical irreducible modules, according to van der Jeugt \cite[Theorem 5.6]{vdJ}, afford a Bernstein-Leites type character formula, and one can readily prove that the formula in loc.~cit.~and the Kac-Wakimoto character formula for $\osp(2|2n)$ (see Section \ref{sec:KW formula}) are indeed equivalent. Thus we shall assume that $m\ge 2$ in the sequel when discussing $\osp(2m|2n)$.

Suppose that $\G=\osp(2m|2n)$ with $m\ge 2$.  We assume that ${\mf l}=\gl(n|m)$ is the subalgebra of $\G$ corresponding to the subdiagram \eqref{ABC:diagram:A1} of \eqref{ABC:diagram:D1}.

In this case, we consider $\ep\delta$-sequences with $m$ $\ep$'s and $n$ $\delta$'s, where we assign $\pm$ sign to the right-most $\ep$ for an $\ep\delta$-sequence ending with a $\delta$. Then the $W$-conjugacy classes of Borel subalgebras ${\mf b}$ of $\G$, or the Borel subalgebras ${\mf b}$ of $\G$ with ${\mf b}_\even=\left({\mf b}^{\tt st}\right)_\even$, are in one-to-one correspondence with such signed $\ep\delta$-sequences \cite[Section 1.3.4]{CW}. Let $s({\mf b})=-1$ if there exists a $-$ in the $\ep\delta$-sequence corresponding to ${\mf b}$, and $s({\mf b})=1$, if there is no $-$ in the sequence. In a way similar to the case of $\osp(2m+1|2n)$, the correspondence can be described as follows: Suppose that we are given such a signed $\ep\delta$-sequence. In the cases when the sequence ends with an $\ep$ or when the sequence ends with a $\delta$ and $s(\mf b)=1$, it gives a fundamental system $\underline{\Pi}$ for ${\mf l}$. In the case when the sequence ends with a $\delta$ and $s(\mf b)=-1$, we note that, ignoring the minus sign, it still gives a fundamental system $\underline{\Pi}'$ for ${\mf l}$. Now, we replace $\ep_m$ with $-\ep_m$ in $\underline{\Pi}'$ and denote this new subset of roots of $\G$ by $\underline{\Pi}$. Then the corresponding fundamental system $\Pi$ for $\G$ is given by
\begin{equation*}
\Pi=
\begin{cases}
\underline{\Pi}\cup \{\ep_{m-1}+\ep_{m}\}, & \text{if the $\ep\delta$-sequence ends with $\ep\ep$},\\
\underline{\Pi}\cup \{\delta_n+\ep_m\}, & \text{if the $\ep\delta$-sequence ends with $\delta\ep$},\\
\underline{\Pi}\cup \{2\delta_n\}, & \text{if the $\ep\delta$-sequence ends with $\delta$},
\end{cases}
\end{equation*}
so that the associated Dynkin diagrams are

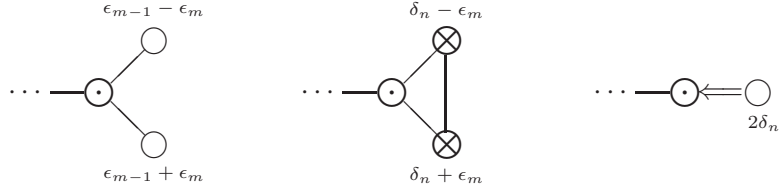
\begin{equation*}
\hskip -1cm\setlength{\unitlength}{0.16in}
\begin{picture}(8,3)
\put(1.45,0.5){\makebox(0,0)[c]{$\cdots$}}
\put(3.8,0.5){\makebox(0,0)[c]{$\bigodot$}}
\put(5.6,2.2){\makebox(0,0)[c]{$\bigcirc$}}
\put(5.6,-1.2){\makebox(0,0)[c]{$\bigcirc$}}
\put(2.2,0.5){\line(1,0){1.1}}
\put(4.2,0.8){\line(1,1){1.1}}
\put(4.2,0.2){\line(1,-1){1.1}}
%\put(2.8,-0.5){\makebox(0,0)[c]{\tiny $\ep_{m-2}-\ep_{m-1}$}}
\put(5.6,-2.2){\makebox(0,0)[c]{\tiny $\ep_{m-1}+\ep_{m}$}}
\put(5.6,3.2){\makebox(0,0)[c]{\tiny $\ep_{m-1}-\ep_{m}$}}
\end{picture}\ \ \ \ \
\begin{picture}(8,3)
\put(1.45,0.5){\makebox(0,0)[c]{$\cdots$}}
\put(3.8,0.5){\makebox(0,0)[c]{$\bigodot$}}
\put(5.6,2.2){\makebox(0,0)[c]{$\bigotimes$}}
\put(5.6,-1.2){\makebox(0,0)[c]{$\bigotimes$}}
\put(2.2,0.5){\line(1,0){1.1}}
\put(4.2,0.8){\line(1,1){1.1}}
\put(4.2,0.2){\line(1,-1){1.1}}
\put(5.55,1.7){\line(0,-1){2.5}}
%\put(2.8,-0.5){\makebox(0,0)[c]{\tiny $\ep_{m-2}-\ep_{m-1}$}}
\put(5.6,-2.2){\makebox(0,0)[c]{\tiny $\delta_{n}+\ep_{m}$}}
\put(5.6,3.2){\makebox(0,0)[c]{\tiny $\delta_{n}-\ep_{m}$}}
\end{picture}\ \ \ \ \
\begin{picture}(8,3)
\put(1.45,0.5){\makebox(0,0)[c]{$\cdots$}}
\put(3.8,0.5){\makebox(0,0)[c]{$\bigodot$}}
\put(6.2,0.5){\makebox(0,0)[c]{$\bigcirc$}}
\put(2.2,0.5){\line(1,0){1.1}}
\put(4.2,0.3){$\Longleftarrow$}
%\put(2.8,-0.5){\makebox(0,0)[c]{\tiny $\ep_{m-2}-\ep_{m-1}$}}
\put(6.4,-0.5){\makebox(0,0)[c]{\tiny $2\delta_{n}$}}
\end{picture}
\end{equation*}
\vskip 1cm
\noindent respectively, where $\bigodot$ is either $\bigcirc$ or $\bigotimes$ depending on $\underline{\Pi}$ or the $\ep\delta$-sequence.

Note that the sequence $\underbrace{\delta\ldots\delta}_n\underbrace{\ep\ldots\ep}_m$  gives rise to the diagram \eqref{ABC:diagram:D1} and hence corresponds to ${\mf b}^{\tt st}$, while the sequence $\underbrace{\ep\ldots\ep}_m\underbrace{\delta\ldots\delta}_n$ gives \vskip 5mm
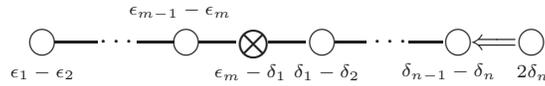
\begin{center}
\vskip -0.3cm
\begin{equation*}\label{ABC:diagram:D5}
\hskip -1cm \setlength{\unitlength}{0.16in}
\begin{picture}(24,1)
\put(5.7,0.5){\makebox(0,0)[c]{$\bigcirc$}}
\put(8.2,0.5){\makebox(0,0)[c]{$\cdots$}}
\put(10.4,0.5){\makebox(0,0)[c]{$\bigcirc$}}
\put(12.6,0.45){\makebox(0,0)[c]{$\bigotimes$}}
\put(14.85,0.5){\makebox(0,0)[c]{$\bigcirc$}}
\put(17.15,0.5){\makebox(0,0)[c]{$\cdots$}}
\put(19.3,0.5){\makebox(0,0)[c]{$\bigcirc$}}
\put(21.7,0.5){\makebox(0,0)[c]{$\bigcirc$}}
\put(6.1,0.5){\line(1,0){1.4}} \put(8.8,0.5){\line(1,0){1.2}}
\put(10.8,0.5){\line(1,0){1.2}} \put(13.1,0.5){\line(1,0){1.2}}
\put(15.28,0.5){\line(1,0){1}} \put(17.7,0.5){\line(1,0){1.1}}
\put(19.7,0.25){$\Longleftarrow$}
\put(5.7,-0.5){\makebox(0,0)[c]{\tiny $\ep_1-\ep_{2}$}}
\put(10.2,1.5){\makebox(0,0)[c]{\tiny $\ep_{m-1}-\ep_{m}$}}
\put(12.5,-0.5){\makebox(0,0)[c]{\tiny $\ep_m-\delta_1$}}
\put(15,-0.5){\makebox(0,0)[c]{\tiny $\delta_1-\delta_2$}}
\put(19,-0.5){\makebox(0,0)[c]{\tiny $\delta_{n-1}-\delta_n$}}
\put(21.75,-0.5){\makebox(0,0)[c]{\tiny $2\delta_{n}$}}
\end{picture}
\end{equation*}
\vskip 0.5cm
\end{center}
and $\underbrace{\ep\ldots\ep(-\ep)}_m\underbrace{\delta\ldots\delta}_n$\ gives \vskip 5mm
\begin{center}
\vskip -0.3cm
\begin{equation*}\label{ABC:diagram:D6}
\hskip -1cm \setlength{\unitlength}{0.16in}
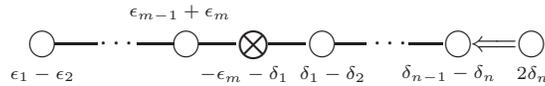
\begin{picture}(24,1)
\put(5.7,0.5){\makebox(0,0)[c]{$\bigcirc$}}
\put(8.2,0.5){\makebox(0,0)[c]{$\cdots$}}
\put(10.4,0.5){\makebox(0,0)[c]{$\bigcirc$}}
\put(12.6,0.45){\makebox(0,0)[c]{$\bigotimes$}}
\put(14.85,0.5){\makebox(0,0)[c]{$\bigcirc$}}
\put(17.15,0.5){\makebox(0,0)[c]{$\cdots$}}
\put(19.3,0.5){\makebox(0,0)[c]{$\bigcirc$}}
\put(21.7,0.5){\makebox(0,0)[c]{$\bigcirc$}}
\put(6.1,0.5){\line(1,0){1.4}} \put(8.8,0.5){\line(1,0){1.2}}
\put(10.8,0.5){\line(1,0){1.2}} \put(13.1,0.5){\line(1,0){1.2}}
\put(15.28,0.5){\line(1,0){1}} \put(17.7,0.5){\line(1,0){1.1}}
\put(19.7,0.25){$\Longleftarrow$}
\put(5.7,-0.5){\makebox(0,0)[c]{\tiny $\ep_1-\ep_{2}$}}
\put(10.2,1.5){\makebox(0,0)[c]{\tiny $\ep_{m-1}+\ep_{m}$}}
\put(12.3,-0.5){\makebox(0,0)[c]{\tiny $-\ep_m-\delta_1$}}
\put(15.2,-0.5){\makebox(0,0)[c]{\tiny $\delta_1-\delta_2$}}
\put(19,-0.5){\makebox(0,0)[c]{\tiny $\delta_{n-1}-\delta_n$}}
\put(21.75,-0.5){\makebox(0,0)[c]{\tiny $2\delta_{n}$}}
\end{picture}
\end{equation*}
\vskip 0.5cm
\end{center}

It is known by \cite{K1} (see also \cite[Theorem 2.14]{CW}) that { given $\gamma\in \h^*$, $\gamma$ is a ${\mf b}^{\tt st}$-highest weight of a finite-dimensional irreducible integer weight $\G$-module $V$ if and only if $\gamma=\la^\natural$ or $\la_-^\natural$ for some $\la\in\mc{H}(n|m)$}, that is, $V\cong  L_{\mf b^{\texttt{st}}}(\G,\la^\natural)=L(\la^\natural)$ or $V\cong  L_{\mf b^{\texttt{st}}}(\G,\la_-^\natural)=L(\la_-^\natural)$. Here we also regard $\la^\natural, \la^\natural_-$ in \eqref{lambda:natural} as weights for $\G$.

Let $\sigma$ be the outer automorphism of $\G$ induced by the diagram automorphism that interchanges $\ep_{m-1}-\ep_m$ and $\ep_{m-1}+\ep_m$ in \eqref{ABC:diagram:D1}.
The diagram symmetry induces a linear map on $\h^\ast$ preserving the standard bilinear form, which we still denote by $\sigma$, such that $\sigma(\delta_k)=\delta_k$ for $1\leq k\leq n$, $\sigma(\ep_l)=\ep_l$ for $1\leq l\leq m-1$, and $\sigma(\ep_m)=-\ep_m$. It also acts in a natural way on the characters of $\G$-modules in the following sense. For a $\G$-module $V$, let us denote by $V^\sigma$ the $\G$-module whose underlying space is $V$ with the action of $\G$ twisted by $\sigma$. Then we have ${\rm ch}V^{\sigma}=\sigma({\rm ch}V)$.

Let $\mf b$ be a Borel subalgebra of $\G$ with $\mf b_{\bar 0}=(\mf b^{\texttt{st}})_\even$.
If the associated $\ep\delta$-sequence ends with an $\ep$, then we have $\sigma({\mf b})={\mf b}$.
If the associated $\ep\delta$-sequence ends with a $\delta$, then $\sigma({\mf b})$ is a Borel subalgebra different from ${\mf b}$, whose $\ep\delta$-sequence is obtained from that of ${\mf b}$ by changing the sign on the right-most $\ep$. One can check that $L_{\mf b}(\G,\gamma)^\sigma\cong L_{\sigma({\mf b})}(\G,\sigma(\gamma))$. In particular, we have for $\la\in\mc{H}(n|m)$,
\begin{equation}\label{eq:sigma twist}
L (\la_-^\natural)\cong L (\la^\natural)^\sigma.
\end{equation}

%Suppose that $s({\mf b})=1$. Then $\mf b$ or $\underline{\mf b}={\mf b}\cap {\mf l}$ determines a unique $\ep\delta$-sequence with no $-$ sign on $\ep$.
For $\la\in\mc{H}(n|m)$, let $\la^{{\mf b}}$ and $\la^{{\mf b}}_-$ denote the ${\mf b}$-highest weight of $L(\la^\natural)$ and $L(\la_-^\natural)$, respectively. By
\cite{CLW} and \cite[Theorem 2.62]{CW}, we have
\begin{equation}\label{eq:aux001}
\begin{cases}
\la^{{\mf b}}=\la^{\underline{\mf b}},\   \  \la_-^{{\mf b}}=\la^{\underline{\mf b}}_-, & \text{if the $\ep\delta$-sequence ends with $\ep$,}\\
\la^{{\mf b}}=\la^{\underline{\mf b}}, & \text{if the $\ep\delta$-sequence ends with $\delta$ and $s(\mf b)=1$},\\
{\la_-^{{\mf b}}=\la_-^{\underline{\mf b}'}}, & \text{if the $\ep\delta$-sequence ends with $\delta$ and $s(\mf b)=-1$,}\\
\end{cases}
\end{equation}
where $\la^{\underline{\mf b}}$ and $\la_-^{\underline{\mf b}}$ are given in \eqref{frob:ht:wt} with $\underline{\mf b}={\mf b}\cap {\mf l}$ for $s({\mf b})=1$, and ${\mf b}'=\sigma({\mf b})$ for $s({\mf b})=-1$. Note that no general formula for $\la^{\mf b}$ (respectively, $\la_-^{\mf b}$) is known when the $\ep\delta$-sequence ends with $\delta$ with $s(\mf b)=-1$ (respectively, $s(\mf b)=1$).

\subsection{Zuckermann functor and Gruson-Serganova's ``typical lemma''}
Suppose that $\G$ is either $\gl(n|m)$ or $\osp(\ell|2n)$ with a Borel subalgebra ${\mf b}$.
Let $\mf p$ be  a parabolic subalgebra with Levi subalgebra $\mf{l}$ and nilradical $\mf{u}$. Let $\mf{u}^-$ be the opposite nilradical so that we have $\G=\mf{p}+\mf{u}^-$. {Let $\Phi^+_{{\mf b}}({\mf l}_\kappa)$ be the set of roots of ${\mf l}_\kappa$ in $\Phi^+_{{\mf b}}$ and $\Phi^+_{{\mf b}}({\mf u}_\kappa)=\Phi^+_{{\mf b},\kappa}\setminus \Phi^+_{{\mf b}}({\mf l}_\kappa)$ for $\kappa=\ov{0}, \ov{1}$.}

Let $\mc{HC}(\G,\mf{l}_\even)$ be the category of $\G$-modules that are direct sums of finite-dimensional simple $\mf{l}_\even$-modules, and let $\mc{HC}(\G,\mf{\G}_{\bar 0})$ be defined similarly.
Let $\mc L_0:\mc{HC}(\G,\mf{l}_\even)\rightarrow \mc{HC}(\G,\G_{\bar 0})$ be the Zuckermann functor defined in \cite[(12)]{San}, which is the same as in \cite[Section 3]{Ser}. We denote by $\mc L_i$ the $i$th derived functor of $\mc L_0$. By the classical Borel-Weil-Bott theorem, it follows that $\mc L_i(M)=0$, for $i\gg 0$, for any $M\in\mc{HC}(\G,\mf l_\even)$.

Let $L(\mf{l},\gamma)$ be a finite-dimensional irreducible $\mf{l}$-module with $\underline{\mf b}(=\mf b\cap\mf l)$-highest weight $\gamma\in\h^*$, which is extended to an irreducible $\mf p$-module in a trivial way.  Let ${\rm Ind}_{\mf p}^{\G}L(\mf{l},\gamma)$ be the induced module, which is a parabolic Verma module over $\G$. The following proposition is well-known.

\begin{prop}\label{prop:same:central}\cite[Section 1.3]{Ger} Let $L(\mf{l},\gamma)$ be a finite-dimensional irreducible $\mf{l}$-module so that ${\rm Ind}_{\mf p}^{\G}L(\mf{l},\gamma)\in\mc{HC}(\G,\mf{l}_\even)$.
\begin{itemize}
\item[(1)] The $\G$-module $\mc L_0\left({\rm Ind}_{\mf p}^{\G}L(\mf{l},\gamma)\right)$ is the maximal finite-dimensional quotient of ${\rm Ind}_{\mf p}^{\G}L(\mf{l},\gamma)$.
\item[(2)] The $\G$-module $\mc L_i\left({\rm Ind}_{\mf p}^{\G}L(\mf{l},\gamma)\right)$ is finite-dimensional for all $i\geq 0$.
    \item[(3)] Let $I$ be the annihilator of ${\rm Ind}_{\mf p}^{\G}L(\mf{l},\gamma)$ in $U(\G)$. Then $I$ annihilates every $\mc{L}_i\left({\rm Ind}_{\mf p}^{\G}L(\mf{l},\gamma)\right)$ for all $i\ge 0$. In particular, all the $\G$-modules $\mc L_i\left({\rm Ind}_{\mf p}^{\G}L(\mf{l},\gamma)\right)$ have the same central character.
\end{itemize}
\end{prop}

For $V\in\mc{HC}(\G,\mf{l})$, we define the {\em Euler characteristic} of $V$ to be the virtual $\G$-module
\begin{align*}
\mc E(V):=\sum_{i=0}^\infty(-1)^i\mc{L}_i(V).
\end{align*}

The following character formula for the Euler characteristic of the parabolic Verma module is known. Indeed, it follows readily from the classical Borel-Weil-Bott theorem and the Weyl character formula for semisimple Lie algebras.

\begin{prop}\label{form:Euler}\cite[Lemma 3.2]{Ser}
Let $M$ be a finite-dimensional $\mf l$-module so that ${\rm Ind}_{\mf p}^{\G}M\in\mc{HC}(\G,\mf{l})$. We have
\begin{align*}
{\rm ch}\mc E({\rm Ind}_{\mf p}^{\G}M)=D^{-1}_{\mf b}\sum_{w\in W}(-1)^{\ell(w)}w\left( \frac{e^{\rho^{\mf b}}{\rm ch}M}{\prod_{\alpha\in\Phi^+_{{\mf b}}(\mf{l_{\bar 1}})}(1+e^{-\alpha})} \right).
\end{align*}
\end{prop}

\begin{rem}\label{form:Euler:1}
Equivalently, by the $W$-invariance of $D^{\mf b}_\odd$, we have
\begin{align*}
{\rm ch}\mc E({\rm Ind}_{\mf p}^{\G}M)=\frac{1}{D_{\mf b,\even}}\sum_{w\in W}(-1)^{\ell(w)}w\left( {e^{\rho_\even^{\mf b}}{\rm ch}M}{\prod_{\alpha\in\Phi^+_{\mf b}(\mf{u_{\bar 1}})}(1+e^{-\alpha})} \right).
\end{align*}
It follows that for two parabolic subalgebras $\mf{p}$ and $\mf q$ with the same even Borel and Levi subalgebras and the same odd nilradical, the Euler characteristics of the parabolic Verma modules coincide, i.e., we have for an $\mf l$-module $M$ $${\rm ch}\mc E({\rm Ind}_{\mf p}^{\G}M)={\rm ch}\mc E({\rm Ind}_{\mf q}^{\G}M).$$
\end{rem}

Let $\gamma\in\h^\ast$ be a ${\mf b}$-highest weight of a finite-dimensional irreducible $\G$-module and let $\chi_{\gamma}$ denote its central character. Following \cite{GS}, we say that $\mf p$ is {\em admissible} for $\chi_{\gamma}$, if for any ${\mf b}$-highest weight $\gamma'$ of a finite-dimensional irreducible $\G$-module  satisfying $\chi_{\gamma}=\chi_{\gamma'}$, we have $(\gamma'+\rho^\mf b,\frac{2\beta}{(\beta,\beta)})\ge 0$ for all $\beta\in\Phi^+_{\mf b}(\mf u_\even)$. The following is Gruson-Serganova's ``typical lemma'', and it plays a fundamental role in their algorithm for computing finite-dimensional irreducible characters over the ortho-symplectic Lie superalgebras \cite{GS}.

\begin{lem}\cite[Lemma 5]{GS}\label{lem:GS}
Let $\mf b$ be a Borel subalgebra and let $\gamma\in\h^\ast$ be a ${\mf b}$-highest weight of a finite-dimensional irreducible $\G$-module. Suppose that
\begin{itemize}
\item[(i)] a parabolic subalgebra $\mf p=\mf l+\mf u$ is admissible for $\chi_{\gamma}$,

\item[(ii)] the Levi subalgebra  $\mf l$ contains a maximal mutually orthogonal set of isotropic odd simple roots orthogonal to $\gamma+\rho^\mf b$ of cardinality equal to the degree of atypicality of $\gamma$,

\item[(iii)] $\left(\gamma+\rho^\mf b,\frac{2\beta}{(\beta,\beta)}\right)> 0$ for all $\beta\in\Phi^+_{\mf b}(\mf u_\even)$.
\end{itemize}
 Then we have
\begin{align*}
\mc L_i\left({\rm Ind}_{\mf p}^\G L(\mf l,\gamma)\right)=
\begin{cases}
L(\G,\gamma),&\text{ if }i=0,\\
0,&\text{ if } i>0.
\end{cases}
\end{align*}
In particular, we have ${\rm ch}\mc{E}{\left({\rm Ind}_{\mf p}^\G L(\mf l,\gamma)\right)}={\rm ch}{L(\G,\gamma)}$.
\end{lem}

\section{Classification of tame modules and the Kac-Wakimoto Conjecture}\label{sec:tame:KW}

\subsection{Tame modules}

Let $\G$ be a finite-dimensional basic Lie superalgebra and let $\h$ be a fixed Cartan subalgebra. Suppose that $V$ is a finite-dimensional irreducible $\G$-module of degree of atypicality $k\ge 0$. Following \cite[Definition 3.5]{KW2}, we say that $V$ is {\em tame} if there exists a Borel subalgebra $\mf b$
%with simple system $\Pi_{\mf b}$ and ${\mf b}$-highest weight $\gamma\in \h^\ast$
such that
\begin{itemize}
\item[(T1)] $V=L_{\mf b}(\G,\gamma)$ for some ${\mf b}$-highest weight $\gamma\in \h^\ast$,
\item[(T2)] there exists a {\em distinguished} set $T_{\gamma}\subseteq\Pi_{\mf b}$ consisting of $k$ mutually orthogonal isotropic odd simple roots satisfying $(\gamma+\rho^{\mf b},\beta)=0$ for all $\beta\in T_{\gamma}$.
\end{itemize}
In this case, we shall also say that $V$ is tame with respect to $\mf b$ of highest weight $\gamma$.

It was shown  in \cite[Theorem 20]{CHR} that an irreducible $\gl(n|m)$-module is tame if and only if it is a Kostant module in the sense of \cite{BS}. The theorem below first appear in \cite[Corollary 2.4]{MJ}. Since irreducible polynomial modules are Kostant modules, it also follows directly from \cite[Theorem 20]{CHR}.

\begin{thm}\label{thm:CHR1}
For $\la\in\mc{H}(n|m)$, the polynomial $\gl(n|m)$-module $L(\gl(n|m),\la^\natural)$ is tame.
\end{thm}

Let $V$ be a finite-dimensional irreducible $\G$-module and suppose that $V=L_{\mf b}(\G,\gamma)$ for some Borel subalgebra $\mf b$ with ${\mf b}$-highest weight $\gamma\in\h^\ast$.
Given $w\in W$, consider the Borel subalgebra $w(\mf b)$. Then the fundamental systems $w(\Pi_{\mf b})$ and $\Pi_{w(\mf b)}$ coincide. The $\gamma$-weight space $V_{\gamma}$ is one-dimensional and we have $V_{\gamma+\alpha}=0$ for any $\alpha\in\Pi_{\mf b}$. Since $V$ is finite-dimensional, the space $V_{w(\gamma)}$ is one dimensional and $V_{w(\gamma)+w(\alpha)}=0$. Thus, we conclude that the ${w(\mf b)}$-highest weight for $V$ is $w(\gamma)$, and $L_{\mf b}(\G,\gamma)\cong L_{w(\mf b)}\left(\G,w\left(\gamma\right)\right)$ as $\G$-modules.

Suppose  in addition that $V$ is tame with respect to $\mf b$. Then, there exists a distinguished set $T_{\gamma}\subseteq\Pi_{\mf b}$ satisfying the condition (T2) above. Since the standard non-degenerate bilinear form on $\h^*$ is $W$-invariant and $w(\rho^{\mf b})=\rho^{w(\mf b)}$, it follows that $w(T_{\gamma})$ is a subset of $\Pi_{w(\mf b)}$ satisfying the condition (T2).  Thus, $V$ is also tame with respect to $w(\mf b)$ of highest weight $w(\gamma)$, where a distinguished set $T_{w(\gamma)}$ is given by $w(T_{\gamma})$. We summarize the above discussion in the following.

\begin{prop}\label{prop:reduction}
Let $\G$ be a finite-dimensional basic Lie superalgebra and let $V$ be a finite-dimensional irreducible $\G$-module.  Then $V$ is tame with respect to a Borel subalgebra $\mf b$ of highest weight $\gamma$ if and only if it is tame with respect to the Borel subalgebra $w(\mf b)$ of highest weight $w(\gamma)$ for any $w\in W$.
\end{prop}

Now we assume that $\G=\osp(\ell|2n)$.
Proposition \ref{prop:reduction} implies that in order to classify the tame $\G$-modules, we need to consider only Borel subalgebras $\mf b$ with $\mf b_{\bar 0}=(\mf b^{\texttt{st}})_\even$. We let ${\mf l}=\gl(n|m)$ be the Levi subalgebra of $\G$ corresponding to the subdiagram \eqref{ABC:diagram:A1} of \eqref{ABC:diagram:B1} and \eqref{ABC:diagram:D1}. Note that $\underline{\mf b}^{\tt st}=\mf b^{\tt st}\cap\mf l$.

\begin{lem}\label{lem:T:set:Atype} Let $\G$ be an ortho-symplectic Lie superalgebra and {let $V$ be a finite-dimensional irreducible $\G$-module of degree of atypicality $k$. Suppose that $\mf b$ and $\mf b'$ are two Borel subalgebras related by $s_\alpha(\Pi_{\mf b})=\Pi_{\mf b'}$, where $\alpha$ is an isotropic odd simple root in $\Pi_{\mf b}$ of the form $\pm(\delta_i-\epsilon_j)$.  Let $\gamma$ and $\gamma'$ be the highest weights of $V$ with respect to ${\mf b}$ and ${\mf b}'$, respectively.}
Suppose that $T_{\mf b}$ is a set of $k$ mutually orthogonal isotropic odd roots of the form $\{\delta_{i_r}-\epsilon_{j_r}\,|\, r=1,\ldots,k\}$ such that $(\gamma+\rho^{\mf b},\delta_{i_r}-\epsilon_{j_r})=0$, for all $r$.  Then there exists a set $T_{\mf b'}$ consisting of $k$ mutually orthogonal isotropic odd roots of the same form orthogonal to $\gamma'+\rho^{\mf b'}$.
\end{lem}

\begin{proof}
The proof consists of considering all possible scenarios for $\alpha$. For definiteness let us suppose that $\alpha=\delta_i-\epsilon_j$. If $\alpha=-\delta_i+\epsilon_j$, then we just replace $\alpha$ by $-\alpha$ below.

Case (1): $(\gamma +\rho^{\mf b},\alpha)\not=0$. Then by \eqref{eq:odd reflection rho} $\gamma +\rho^{\mf b}=\gamma'+\rho^{\mf b'}$ and we may take the same set $T_{\mf b}$ for $T_{\mf b'}$.

Case (2): $(\gamma +\rho^{\mf b},\alpha)=0$ and $\alpha\in T_{\mf b}$. Then by \eqref{eq:odd reflection rho} $\gamma+\rho^{\mf b}+\alpha=\gamma'+\rho^{\mf b'}$ and we still can take $T_{\mf b'}$ to be $T_{\mf b}$, since the roots in $T_{\mf b}$ are mutually orthogonal.

Case (3): $(\gamma +\rho^{\mf b},\alpha)=0$ and $\alpha\not\in T_{\mf b}$. In this case there are three possible scenarios.   In the Subcase (i) we have a $\beta\in T_{\mf b}$ such that $\beta=\delta_i-\epsilon_{j'}$, $j'\not=j$, and furthermore there is no $\beta'\in T_{\mf b}$ of the form $\delta_{i'}-\epsilon_j$. In this case, we take $T_{\mf b'}=\left(\{\alpha\}\cup T_{\mf b}\right)\setminus\{\beta\}$. In the Subcase (ii) we have a $\beta\in T_{\mf b}$ such that $\beta=\delta_{i'}-\epsilon_j$, $i'\not=i$, and furthermore there is no $\beta'\in T_{\mf b}$ of the form $\delta_i-\epsilon_{j'}$.  Here we take $T_{\mf b'}=\left(\{\alpha\}\cup T_{\mf b}\right)\setminus\{\beta\}$. In the Subcase (iii) we have a $\beta\in T_{\mf b}$ of the form $\beta=\delta_i-\epsilon_{j'}$, ${j'}\not=j$, and we also have a $\beta'$ of the form $\beta'=\delta_{i'}-\epsilon_j$, $i'\not=i$. In this case it is also easy to see that we can replace $\beta$ and $\beta'$ in $T_{\mf b}$ by $\alpha$ and $\delta_{i'}-\epsilon_{j'}$ to get $T_{\mf b'}$.
\end{proof}

Suppose that $\G=\osp(2m+1|2n)$. By our discussion in Section \ref{sec:Borels B}, the Borel subalgebras $\mf b$ with $\mf b_{\bar 0}=(\mf b^{\texttt{st}})_\even$ are in one-to-one correspondence with the Borel subalgebras $\underline{\mf b}$ of  $\mf l$ such that $\underline{\mf b}_\even=\left(\underline{\mf b}^{\tt st}\right)_\even$, which are obtained by $\underline{\mf b}=\mf b\cap\mf l$ or removing the right-most node of the corresponding Dynkin diagram for ${\mf b}$. %These are the Borel subalgebras $\underline{\mf b}$ of ${\mf l}$ such that $\underline{\mf b}_\even=\left(\underline{\mf b}^{\tt st}\right)_\even$.
We observe that the difference between the Weyl vectors $\rho^{\mf b}$ for $\osp(2m+1|2n)$ and $\rho^{\underline{\mf b}}$ for $\gl(n|m)$ is equal to a scalar multiple of the supertrace ${\bf 1}_{n|m}=\sum_{i=1}^n\delta_i-\sum_{j=1}^m\ep_j$ for $\mf l$, and hence $(\rho^{\mf b}-\rho^{\underline{\mf b}},\delta_i-\ep_j)=0$ for all $i,j$.

\begin{thm}\label{thm:classification B}
Let $\G=\osp(2m+1|2n)$  and let $\la\in\mc H(n|m)$. Suppose that $L(\la^\natural)$ has degree of atypicality $k\geq 1$.  Then  $L(\la^\natural)$ is tame if and only if it satisfies the following condition:
\begin{itemize}
\item[(\underline{T})] there exist $k$ mutually orthogonal isotropic odd roots of the form $\beta_{r}=\delta_{i_r}-\ep_{j_r}$ such that $(\la^\natural+\rho^{\mf b^\st},\beta_r)=0$, for $r=1,\ldots,k$.
\end{itemize}
\end{thm}

\begin{proof}
Suppose that $L(\la^\natural)$ is tame with respect to $\mf b$ of highest weight $\la^{\mf b}$ with a distinguished subset $T_{\la^{\mf b}}\subseteq \Pi_{\mf b}$. By Proposition \ref{prop:reduction}, we may assume that $\mf b_\even=(\mf b^{\texttt{st}})_\even$. By the classification of Borel subalgebras in Section \ref{sec:Borels B}, all the simple odd roots in $\Pi_{\mf b}$ are of the form $\pm(\delta_i-\ep_j)$, and thus, so are the simple roots in $T_{\la^{\mf b}}$.

Since $\mf b$ is a Borel subalgebra of $\G$ with $\mf b_\even=(\mf b^\st)_\even$, applying a sequence of odd reflections corresponding to odd simple roots of the form $\pm(\delta_i-\ep_j)$, we transform $\mf b$ to the standard Borel subalgebra ${\mf b}^\st$. By Lemma \ref{lem:T:set:Atype} we obtain a set of $k$ mutually orthogonal isotropic roots of the form $\delta_i-\ep_j$ orthogonal to $\la^\natural+\rho^{{\mf b}^\st}$, and hence Condition (\underline{T}) is satisfied.

Conversely, suppose that Condition (\underline{T}) holds. {We first note that $L({\mf l},\la^\natural)$ has the degree of atypicality $k$ since $\rho^{{\mf b}^{\tt st}}-\rho^{\underline{\mf b}^{\tt st}}$ is a multiple of ${\bf 1}_{n|m}$ and hence $(\la^\natural + \rho^{{\mf b}^{\tt st}},\delta_i-\ep_j)=0$ if and only if $(\la^\natural + \rho^{\underline{\mf b}^{\tt st}},\delta_i-\ep_j)=0$ for $i,j$. By Theorem \ref{thm:CHR1} there exists a Borel subalgebra $\underline{\mf b}$ of $\mf l$ with the property that $L({\mf l},\la^\natural)$ is tame with respect to $\underline{\mf b}$ of highest weight $\la^{\underline{\mf b}}$ with a distinguished set $T$ of $k$ isotropic odd simple roots. Since $\underline{\mf b}_\even=\left(\underline{\mf b}^{\tt st}\right)_\even$, $\underline{\mf b}$ corresponds uniquely to a Borel subalgebra $\mf b$ of $\G$  with $\mf b_\even=(\mf b^\st)_\even$ and we further have $T\subseteq \Pi_{\mf b}$.}
Since $\rho^{\mf b}-\rho^{\underline{\mf b}}$ is a multiple of ${\bf 1}_{n|m}$ and  $(\la^{\underline{\mf b}}+\rho^{\underline{\mf b}},\beta)=0$ for all $\beta\in T$, we also have $(\la^{\mf b}+\rho^{\mf b},\beta)=0$ for all $\beta\in T$. Thus, $L(\la^\natural)$ is tame with respect to $\mf b$ of highest weight $\la^{\mf b}$ with the same $T$ as a distinguished set.
\end{proof}

Next, suppose that $\G=\osp(2m|2n)$ with $m\ge 2$. We first observe the following.
\begin{lem}\label{lem:sigma tame}
Let $V$ be a finite-dimensional irreducible $\G$-module.  Then $V$ is tame with respect to ${\mf b}$ of highest weight $\gamma$ if and only if $V^{\sigma}$ is tame with respect to $\sigma({\mf b})$ of highest weight $\sigma(\gamma)$. In this case, we have $T_{\sigma(\gamma)}=\sigma(T_{\gamma})$.
\end{lem}
\begin{proof}
It follows from the fact that  for a given Borel subalgebra ${\mf b}$, $\sigma(\Pi_{\mf b})=\Pi_{\sigma({\mf b})}$ and $\sigma(\rho^{\mf b})=\rho^{\sigma({\mf b})}$.
\end{proof}

In light of Lemma \ref{lem:sigma tame} and \eqref{eq:sigma twist} it is enough to determine when the module $L(\la^\natural)$ is tame, for $\la\in\mc{H}(n|m)$.

\begin{thm}\label{thm:classification D}
Let $\G=\osp(2m|2n)$ with $m\ge 2$, and let $\la\in\mc H(n|m)$. Suppose that $L(\la^\natural)$ has degree of atypicality $k\geq 1$. Then  $L(\la^\natural)$ is tame if and only if one of the following conditions hold:
\begin{itemize}
\item[(i)] In the case when $(\la^\natural,\ep_m)=0$, or equivalently when $\la_{n+1}<m$, Condition {\rm (\underline{T})} holds.
\item[(ii)] In the case when $(\la^\natural,\ep_m)>0$, or equivalently when $\la_{n+1}=m$, we have $k=1$ and there exists an isotropic odd root of the form $\beta=\delta_{i}+\ep_{m}$ for some $1\leq i\leq n$ such that $(\la^\natural+\rho^{\mf b^\st},\beta)=0$,
\end{itemize}
\end{thm}
\begin{proof}
Suppose that $(\la^\natural,\ep_m)=0$ and hence $\la^\natural =\la^\natural_-$.

If $L(\la^\natural)$ is tame with respect to a Borel subalgebra corresponding to an $\epsilon\delta$-sequence ending with an $\epsilon$, then the condition $\la_{n+1}<m$ implies that $(\la^{\mf b}+\rho^{\mf b},\delta_n+\epsilon_m)=0$ if and only if $(\la^{\mf b}+\rho^{\mf b},\delta_n-\epsilon_m)=0$. Thus, if $\delta_n+\epsilon_m$ is in the distinguished set, we can always replace it by $\delta_n-\epsilon_m$, and still get a distinguished set. Thus,  by the classification of Borel subalgebras in Section \ref{sec:Borels D} we can find a distinguished set consisting of odd roots of the form $\pm(\delta_i-\ep_j)$. Since we can transform such a $\mf b$ to $\mf b^{\texttt{st}}$ via a sequence of odd reflections corresponding to $\pm(\delta_i-\ep_j)$, we can apply Lemma \ref{lem:T:set:Atype} to prove that Condition $(\underline{T})$ is satisfied. Now, the equivalence of Condition $(\underline{T})$ is established as in the proof of Theorem \ref{thm:classification B}.

On the other hand, if $L(\la^\natural)$ is tame with respect to a Borel subalgebra corresponding to an $\epsilon\delta$-sequence ending with a $\delta$, it follows from \eqref{eq:sigma twist} and Lemma \ref{lem:sigma tame} that $L(\la^\natural)$ is tame if and only if it is tame with respect to a Borel subalgebra ${\mf b}$ with $s({\mf b})=1$. We can again transform such a $\mf b$ to $\mf b^{\texttt{st}}$ via a sequence of odd reflections as in Lemma \ref{lem:T:set:Atype}. Thus, the same argument as above is again applicable.

Next, consider the case $(\la^\natural,\ep_m)>0$. Since $\la_{n+1}= m$, we have $(\la^\natural + \rho^{{\mf b}^{\texttt{st}}},\delta_i-\epsilon_j)\neq 0$ for all $i, j$.

Suppose first that $L(\la^\natural)$ is tame with respect to a Borel subalgebra $\mf b$ corresponding to an $\epsilon\delta$-sequence ending with an $\epsilon$. Since $\mf b$ is obtained from $\mf b^{\texttt{st}}$ by a sequence of odd reflections corresponding to $\pm(\delta_i-\epsilon_j)$, and $(\la^\natural + \rho^{{\mf b}^{\texttt{st}}},\delta_i-\epsilon_j)\neq 0$ for all $i, j$, we see that  $\la^{\mf b}+\rho^{\mf b}=\la^\natural+\rho^{\mf b^{\texttt{st}}}$, and hence $(\la^{\mf b} + \rho^{{\mf b}},\delta_i-\epsilon_j)\neq 0$, for all $i, j$. This necessarily implies that the $\epsilon\delta$-sequence for $\mf b$ ends with $\delta\epsilon$, and $(\la^{\mf b}+\rho^{\mf b},\delta_n+\epsilon_m)=0$ with $k=1$. Hence $(\la^\natural + \rho^{{\mf b}^{\tt st}}, \delta_n+\ep_m)=(\la^{\mf b}+\rho^{\mf b},\delta_n+\epsilon_m)=0$, establishing (ii) in this case.

%Since $(\la^{\mf b}+\rho^{\mf b},\delta_l-\epsilon_k)\not=0$, for $k=1,\ldots,m-1$ and $l=1,\ldots,n$, it is not hard to see, using \eqref{eq:odd reflection rho}, that we have $(\la^{\mf b^{\texttt{st}}}+\rho^{\mf b^{\texttt{st}}},\delta_i+\epsilon_m)=0$, for some $i$, establishing (ii) in this case.

Next, suppose that $L(\la^\natural)$ is tame with respect to a Borel subalgebra ${\mf b}$ corresponding to an $\epsilon\delta$-sequence ending with a $\delta$. If $s({\mf b})=1$, then by the same argument as in the case of $\la_{n+1}<m$, there exists an isotropic root $\beta$ of the form $\delta_i-\epsilon_j$, which is orthogonal to $\la^\natural + \rho^{{\mf b}^{\texttt{st}}}$. This is a contradiction. So, $L(\la^\natural)$ is tame with respect to a Borel subalgebra ${\mf b}$ with $s({\mf b})=-1$. By Lemma \ref{lem:sigma tame}, $L(\la_-^\natural)$ is tame with respect to ${\mf b}'=\sigma({\mf b})$ with $s({\mf b}')=1$, where any simple odd root in $T_{\la^{{\mf b}'}_-}$ is of the form $\pm(\delta_i-\ep_j)$.
Since $s(\mf b')=1$, we can use a sequence of odd reflections with respect to odd simple roots of the form $\pm(\delta_i-\epsilon_j)$ to transform $\mf b'$ into $\mf b^{\texttt{st}}$. By Lemma \ref{lem:T:set:Atype}, since $\la^{{\mf b}'}_-+\rho^{\mf b'}$ has a set of orthogonal isotropic roots (of cardinality equal to the degree of atypicality) of the form $\{\pm(\delta_i-\epsilon_j)\}$, the same is true for the shifted weight $\la^\natural_-+\rho^{{\mf b}^\texttt{st}}$. Since $\mf b'$ is a Borel subalgebra of $\G$ with $\mf b'_\even=(\mf b^\st)_\even$, and $(\la_-^\natural + \rho^{{\mf b}^{\texttt{st}}},\delta_i)<0$ for $1\leq i\leq n$ and $(\la_-^\natural + \rho^{{\mf b}^{\texttt{st}}},\ep_j)>0$ for $1\leq j\leq m-1$, we have $T_{\la^{{\mf b}'}_-}=\{\pm(\delta_i-\ep_m)\}$ for some $i$, and in particular, the degree of atypicality is $1$. Hence Condition (ii) holds.

Similarly, we can show that Condition (ii) is sufficient for $L(\la^\natural)$ to be tame.
\end{proof}

{Based on \cite{SV2}, we put for $\lambda\in \mc{H}(n|m)$
\begin{equation}\label{eq:e(lambda)}
\begin{split}
i(\lambda')&= \max\{\,i\,|\, 1\leq i\leq m,\ \lambda'_i-i+m-n\geq 0\,\},\\
i^\ast(\lambda')&= \max\{\,i\,|\,1\leq i\leq m,\ \lambda'_i-i+m-n> 0\,\}, \\
e(\lambda)&= i(\lambda')-i^\ast(\lambda').
\end{split}
\end{equation}
Here we assume that $i(\la')=0$ and $i^\ast(\la')=0$ if there is no $i$ such that $\lambda'_i-i+m-n\geq 0$ and} $\lambda'_i-i+m-n> 0$, {respectively. It is straightforward to check that $e(\lambda)=0$ or $1$.}

\subsection{Kac-Wakimoto conjecture}\label{sec:KW formula}

The following conjecture of Kac and Wakimoto goes back to \cite{KW1}, and is explicitly stated in \cite[Conjecture 3.6]{KW2}.

\begin{conjecture}\label{conj:KW}(Kac-Wakimoto)
{\it Let $\G$ be a finite-dimensional basic Lie superalgebra and let $V$ be a finite-dimensional irreducible $\G$-module that is tame with respect to a Borel subalgebra $\mf b$ of highest weight $\gamma$ with a distinguished set $T_{\gamma}\subseteq\Pi_{\mf b}$.
Then we have $V=L_{\mf b}(\G,\gamma)$ and there exists an integer $j_{\gamma}$ such that the following character formula holds for $V$:
\begin{align}\label{formula:kw}
{\rm ch}V=\frac{1}{j_{\gamma}} D_{\mf b}^{-1} \sum_{w\in W}(-1)^{\ell(w)} w\left( \frac{e^{{\gamma}+\rho^{\mf b}}}{\prod_{\beta\in T_{\gamma}}(1+e^{-\beta})} \right).
\end{align}}
\end{conjecture}

It is well-known that when $V$ is typical, Formula \eqref{formula:kw} holds for arbitrary Borel subalgebra ${\mf b}$ with $T_{\gamma}=\emptyset$ and $j_{\gamma}=1$ by a classical result of Kac \cite[Proposition 2.8]{K2}.
In the case when $\G=\gl(n|m)$, Conjecture \ref{conj:KW} was established by Chmutov, Hoyt, and Reif in \cite[Theorem 42]{CHR}. A main ingredient in their proof is the character formula by Su and Zhang for Kostant modules \cite[Corollary 4.13]{SZ1}.  In particular, the validity of the conjecture in the $\gl(n|m)$-case and Theorem \ref{thm:CHR1} imply the following.

\begin{thm}\label{thm:CHR2}\cite{CHR}
For $\la\in\mc{H}(n|m)$, Formula \eqref{formula:kw} holds for $L(\gl(n|m),\la^\natural)$, which is tame with respect to a Borel subalgebra ${\mf b}$ of highest weight $\la^{\mf b}$ with a distinguished set $T_{\la^\mf b}$. Moreover, $j_{\la^{{\mf b}}}=|T_{\la^{{\mf b}}}|!$.
\end{thm}

Now, let us consider the case when $\G=\osp(\ell|2n)$. We fix a $\la\in\mc H(n|m)$ such that $L(\la^\natural)$ is tame of degree of atypicality $k\geq 1$.

Suppose that $\G=\osp(2m+1|2n)$. Let
\begin{equation}\label{eq:borel levi odd-1}
\mf b=\mf b^{\texttt{odd}}, \ \ \ \mf{l}=\h+\osp(2k+1|2k), \ \ \ \underline{\mf b}=\mf b\cap \mf l,
\end{equation}
where  $\mf b^{\texttt{odd}}$ is the Borel subalgebra given in Section \ref{sec:osp def}, and $\osp(2k+1|2k)$ is the subalgebra corresponding to the subdiagram of the Dynkin diagram for $\mf b^{\texttt{odd}}$ consisting of the right-most $2k$ nodes as in \cite{GS,SV2}. The weight $\la^\mf b$ may now be regarded as a weight for the reductive Lie superalgebra $\mf l$.

\begin{lem}\label{lem:htwt:bodd-1}
Let $\G=\osp(2m+1|2n)$ and let $\la\in\mc{H}(n|m)$ such that $L(\la^\natural)$ is tame of degree of atypicality $k\ge 1$.
Under the hypothesis \eqref{eq:borel levi odd-1}, we have the following:
\begin{itemize}
\item[(1)] We have  $$\la^\mf b+\rho^\mf b=\sum_{i=1}^{n-k}a_i\delta_i+\sum_{j=1}^{m-k}b_j\ep_j- \sum_{i=m-k+1}^m\hf\ep_i+\sum_{j=n-k+1}^n\hf\delta_j$$
for some $a_1>\cdots>a_{n-k}>0$ and $b_1>\ldots>b_{m-k}>0$.

\item[(2)] $L(\mf l,\la^{{\mf b}})$ is one-dimensional, which is the trivial module when restricted to the subalgebra $\osp(2k+1|2k)$.
\item[(3)] $L(\la^\natural)$ is tame with respect to ${\mf b}$ with a distinguished set
\begin{equation*}\label{eq:T_lambda-1}
T_{\la^{{\mf b}}}=\{\,\ep_{m-k+i}-\delta_{n-k+i}\,|\, i=1,\ldots,k\,\}.
\end{equation*}
Furthermore, $T_{\la^{{\mf b}}}$ is also a distinguished set for the tame module $L(\mf l,\la^{{\mf b}})$ with respect to $\underline{\mf b}$.
\end{itemize}
\end{lem}

\begin{proof}
{Computing $\rho^{\mf b}$ by \eqref{eq:odd reflection rho} and then using \eqref{eq:aux004}, we see from Theorem \ref{thm:classification B} and \cite[Proposition 3.21]{CW} that (1) holds.}  In particular, (1) implies (2) (see also \cite[p.~4306]{SV2}) and (3).
\end{proof}

Suppose that $\G=\osp(2m|2n)$. If $\la_{n+1}<m$, then we put {
\begin{equation}\label{eq:borel levi odd-2}
\mf b=\mf b^{\texttt{odd}}, \ \ \
\mf{l}=
\begin{cases}
\h+\osp(2k|2k),& \text{if $e(\la)=0$,} \\
\h+\osp(2k+2|2k), & \text{if $e(\la)=1$,}
\end{cases}
\ \ \  \underline{\mf b}=\mf b\cap \mf l,
\end{equation}
where  $\mf b^{\texttt{odd}}$ is as in Section \ref{sec:prelim}, and $\osp(2k|2k)$ (respectively, $\osp(2k+2|2k)$) is the subalgebra corresponding to the subdiagram with the right-most $2k$ (respectively, $2k+1$) nodes of the Dynkin diagram for $\mf b^{\texttt{odd}}$ \cite{GS,SV2}.}

If $\la_{n+1}=m$, then by Theorem \ref{thm:classification D} there exists an isotropic root of the form $\beta=\delta_{i}+\ep_{m}$ for some $1\leq i\leq n$ such that $(\la^\natural+\rho^{\mf b^\st},\beta)=0$. We take $\mf b$ to be the Borel subalgebra corresponding to the following Dynkin diagram:
\begin{center}
\vskip -0.3cm
\begin{equation}\label{ABC:diagram:D7}
\hskip -2.5cm \setlength{\unitlength}{0.16in}
\begin{picture}(24,1)
\put(3.2,0.5){\makebox(0,0)[c]{$\bigcirc$}}
\put(5.7,0.5){\makebox(0,0)[c]{$\bigcirc$}}
\put(8.2,0.5){\makebox(0,0)[c]{$\cdots$}}
\put(10.4,0.5){\makebox(0,0)[c]{$\bigotimes$}}
\put(12.6,0.45){\makebox(0,0)[c]{$\bigcirc$}}
\put(14.85,0.5){\makebox(0,0)[c]{$\bigcirc$}}
\put(17.15,0.5){\makebox(0,0)[c]{$\cdots$}}
\put(19.3,0.5){\makebox(0,0)[c]{$\bigotimes$}}
\put(24.2,0.5){\makebox(0,0)[c]{$\cdots$}}
\put(21.8,0.5){\makebox(0,0)[c]{$\bigotimes$}}
\put(26.6,0.5){\makebox(0,0)[c]{$\bigcirc$}}
\put(28.8,0.5){\makebox(0,0)[c]{$\bigcirc$}}
\put(3.6,0.5){\line(1,0){1.6}}
\put(6.1,0.5){\line(1,0){1.4}} \put(8.8,0.5){\line(1,0){1.2}}
\put(10.8,0.5){\line(1,0){1.4}} \put(13.1,0.5){\line(1,0){1.2}}
\put(15.28,0.5){\line(1,0){1}} \put(17.7,0.5){\line(1,0){1.1}}
\put(19.7,0.5){\line(1,0){1.55}}
\put(24.85,0.5){\line(1,0){1.3}}
\put(22.2,0.5){\line(1,0){1.3}}
\put(26.95,0.25){$\Longleftarrow$}
\put(3,-0.5){\makebox(0,0)[c]{\tiny $\delta_1-\delta_{2}$}}
\put(5.7,-0.5){\makebox(0,0)[c]{\tiny $\delta_2-\delta_{3}$}}
\put(10.3,1.5){\makebox(0,0)[c]{\tiny $\delta_{i-1}-\ep_{1}$}}
\put(12.3,-0.5){\makebox(0,0)[c]{\tiny $\ep_1-\ep_2$}}
\put(15.2,-0.5){\makebox(0,0)[c]{\tiny $\ep_1-\ep_2$}}
\put(19,-0.5){\makebox(0,0)[c]{\tiny $\ep_{m-1}-\delta_i$}}
\put(21.8,1.5){\makebox(0,0)[c]{\tiny $\delta_{i}+\ep_m$}}
\put(26.2,-0.5){\makebox(0,0)[c]{\tiny $\delta_{n-1}-\delta_{n}$}}
\put(28.75,-0.5){\makebox(0,0)[c]{\tiny $2\delta_{n}$}}
\end{picture}
\end{equation}
\vskip 0.5cm
\end{center}
That is, the corresponding numbered $\ep\delta$-sequence is of the form
\begin{align*}
\delta_1\delta_2\cdots\delta_{i-1}\ep_1\ep_2\cdots\ep_{m-1}\delta_i(-\ep_m)\delta_{i+1}\cdots\delta_n.
\end{align*}
We let
\begin{align}\label{eq:borel levi odd-3}
\mf{l}=\h+\mf{sl}(1|1),\quad \underline{\mf b}=\mf b\cap\mf l,
\end{align}
where $\mf{sl}(1|1)$ is the Lie subalgebra of $\G$ generated by the root vector corresponding to $\pm(\delta_i+\ep_m)$. %Let $\mf{p}=\mf{l}+\mf{b}$.

\begin{lem}\label{lem:htwt:bodd-2}
Let $\G=\osp(2m|2n)$, $m\ge 2$, and let $\la\in\mc{H}(n|m)$ such that $L(\la^\natural)$ is tame of degree of atypicality $k\geq 1$.
Under the hypothesis \eqref{eq:borel levi odd-2} and \eqref{eq:borel levi odd-3}, we have the following.
\begin{itemize}
\item[(1)] If $\la_{n+1}<m$, then we have
$$\la^\mf b+\rho^\mf b=\sum_{i=1}^{n-k}a_i\delta_i+\sum_{j=1}^{m-k}b_j\ep_j,$$
{for some $a_1>\cdots>a_{n-k}> 0$ and $b_1>\ldots>b_{m-k}\geq 0$. Furthermore, $b_{m-k}=0$ if and only if $e(\la)=1$.}

\noindent If $\la_{n+1}=m$, then
$$\la^\mf b+\rho^\mf b=\sum_{i=1}^{n}a_i\delta_i+\sum_{j=1}^{m}b_j\ep_j, $$
for some $a_1>\cdots>a_{n}>0$ and $b_1>\ldots>b_{m}>0$ and $a_i=b_m$.

\item[(2)] $L(\mf l,\la^{{\mf b}})$ is one-dimensional, and it is the trivial module when restricted to the subalgebras $\osp(2k+2e(\la)|2k)$ and $\mf{sl}(1|1)$ in the cases when $\la_{n+1}<m$ and $\la_{n+1}=m$, respectively.
\item[(3)] $L(\la^\natural)$ is tame with respect to ${\mf b}$ with a distinguished set
\begin{equation*}\label{eq:T_lambda-2}
T_{\la^{{\mf b}}}=
\begin{cases}
\{\,\delta_{n-k+i}-\ep_{m-k+i}\,|\, i=1,\ldots,k\,\}, &\text{if $\la_{n+1}<m$}, \\
\{\,\delta_i + \ep_{m}\,\}, &\text{if $\la_{n+1}=m$}.
\end{cases}
\end{equation*}
Furthermore, $T_{\la^{{\mf b}}}$ is also a distinguished set for the tame module $L(\mf l,\la^{{\mf b}})$ with respect to $\underline{\mf b}$.
\end{itemize}
\end{lem}
\begin{proof}
Suppose that $\la_{n+1}<m$. {By the same arguments as in the proof of  Lemma \ref{lem:htwt:bodd-1}(1), we have
$$\la^\mf b+\rho^\mf b=\sum_{i=1}^{n-k}a_i\delta_i+\sum_{j=1}^{m-k}b_j\ep_j,$$
for some $a_1>\cdots>a_{n-k}> 0$ and $b_1>\ldots>b_{m-k}\geq 0$. Furthermore, since $(\rho^{\mf b},\ep_{m-k})=0$, we can check that  $b_{m-k}=0$ if and only if} $e(\la)=1$.
The proof of the other parts is a straightforward verification and we omit it.
\end{proof}

\begin{prop}\label{prop:Euler=KW}
Let  $\G=\osp(\ell|2n)$ and let $\mf p={\mf l}+{\mf b}$ be the parabolic subalgebra with Borel subalgebra ${\mf b}$ and Levi subalgebra ${\mf l}$ as in \eqref{eq:borel levi odd-1}, \eqref{eq:borel levi odd-2}, and \eqref{eq:borel levi odd-3}. Suppose that  Formula \eqref{formula:kw} holds for the (trivial) ${\mf l}$-module $L(\mf l,\la^{{\mf b}})$ with respect to $\underline{\mf b}$. Then we have
\begin{align*}
{\rm ch}\mc E({\rm Ind}_{\mf p}^{\G} L(\mf l,\la^{\mf b}))=\frac{1}{j_{\la^{{\mf b}}}}D_{\mf b}^{-1}\sum_{w\in W_\G}(-1)^{\ell(w)} w\left( \frac{e^{\la^\mf b+\rho^{\mf b}}}{\prod_{\beta\in T_{{\la}^\mf b}}(1+e^{-\beta})} \right).
\end{align*}
\end{prop}

\begin{proof}
To simplify notations, we put $\la=\la^{\mf b}$, $\rho=\rho^{\mf b}$, $\underline{\rho}=\rho^{\underline{\mf b}}$, $D=D_{\mf b}$, $\underline{D}=D_{\underline{\mf b}}$, $W=W_\G$, and $T=T_{\la^{{\mf b}}}$. Also, we let $\underline{D}_\kappa=D_{\underline{\mf b},\kappa}$, $\rho_{\kappa}=\underline{\rho}_{\kappa}+\rho'_{\kappa}$, for $\kappa=\even,\odd$, and $j=j_{\la^{\mf b}}$.

{\allowdisplaybreaks
We compute by Proposition \ref{form:Euler}
\begin{align*}
&D\,{\rm ch}\mc E({\rm Ind}_{\mf p}^{\G} L(\mf l,\la))\\ &=  \sum_{w\in W}(-1)^{\ell(w)}w\!\left( \frac{e^\rho{\rm ch}L(\mf l,\la)}{\prod_{\gamma\in \Phi^+(\mf l_\odd)}(1+e^{-\gamma})} \right) \\
&=  \sum_{\sigma\in W/W_{\mf l}}(-1)^{\ell(\sigma)}\sigma\!\left(
\sum_{\tau\in W_{\mf l}}(-1)^{\ell(\tau)}\tau\!\left(
\frac{e^\rho{\rm ch}L(\mf l,\la)}{\prod_{\gamma\in \Phi^+(\mf l_\odd)}(1+e^{-\gamma})} \right)\right)\allowdisplaybreaks\\
&=  \sum_{\sigma\in W/W_{\mf l}}(-1)^{\ell(\sigma)}\sigma\!\left( {\rm ch}L(\mf l,\la)
\sum_{\tau\in W_{\mf l}}(-1)^{\ell(\tau)}\tau\!\left(
\frac{e^{\rho_\even-\underline{\rho}_\odd-\rho'_\odd}}{\prod_{\gamma\in \Phi^+(\mf l_\odd)}(1+e^{-\gamma})} \right)\right) \\
&=  \sum_{\sigma\in W/W_{\mf l}}(-1)^{\ell(\sigma)}\sigma\!\left( \frac{{\rm ch}L(\mf l,\la)}{\underline{D}_\odd}
\sum_{\tau\in W_{\mf l}}(-1)^{\ell(\tau)}\tau\!\left(
{e^{\rho_\even-\rho'_\odd}} \right)\right) \\
&=  \sum_{\sigma\in W/W_{\mf l}}(-1)^{\ell(\sigma)}\sigma\!\left( \frac{{\rm ch}L(\mf l,\la)}{\underline{D}_\odd} e^{\rho'_\even-\rho'_\odd}
\sum_{\tau\in W_{\mf l}}(-1)^{\ell(\tau)}\tau\left(
{e^{\underline{\rho}_\even}} \right)\right) \\
&= \frac{1}{j}  \sum_{\sigma\in W/W_{\mf l}}(-1)^{\ell(\sigma)}\sigma\!\left( \frac{1}{\underline{D}_\even} \sum_{\tau\in W_{\mf l}}(-1)^{\ell(\tau)}\tau\!\left(
\frac{e^{\la+\underline{\rho}}}{\prod_{\beta\in T}(1+e^{-\beta})} \right)
e^{\rho'_\even-\rho'_\odd}
\sum_{\tau\in W_{\mf l}}(-1)^{\ell(\tau)}\tau(e^{\underline{\rho}_\even})\right) \\
&= \frac{1}{j} \sum_{\sigma\in W/W_{\mf l}}(-1)^{\ell(\sigma)}\sigma\!\left( \sum_{\tau\in W_{\mf l}} (-1)^{\ell(\tau)} \tau\left( \frac{e^{\la+\rho}}{\prod_{\beta\in T}(1+e^{-\beta})} \right)\right),
\end{align*}}
\noindent \hskip -1.5 mm where in the last identity we have used the Weyl denominator identity for $\mf l_\even$.  Now the last expression proves our claim.
\end{proof}

\section{Properties of highest weights of tame modules}\label{sec:hw:tame}
In this section, we suppose that $\G=\osp(\ell|2n)$.
Let us denote by $\preccurlyeq$ the partial ordering on the set of $\mf b^{\texttt{st}}$-highest weights of finite-dimensional irreducible $\G$-modules induced by $\Phi^+_{\mf b^{\texttt{st}}}$. That is, for two such weights $\gamma$ and $\nu$ one has $\gamma\preccurlyeq\nu$ if and only if $\nu-\gamma\in\Z_+\Phi^+_{\mf b^{\texttt{st}}}$.

\begin{prop}\label{prop:central ch of tame}
Let $\la,\mu\in \mc{H}(n|m)$ such that  $L(\la^\natural)$ and $L(\mu^\natural)$ are tame. Let $k$ be the degree of atypicality of $L(\la^\natural)$.
\begin{itemize}
\item[(1)] Let $\G=\osp(2m+1|2n)$. Then $\chi_{\la^\natural}=\chi_{\mu^\natural}$ implies $\la^\natural=\mu^\natural$.
\item[(2)] Let $\G=\osp(2m|2n)$ with $m\ge 2$ and suppose that $k\not=1$.
Then
 $\chi_{\la^\natural}=\chi_{\mu^\natural}$, equivalently $\chi_{\la^\natural}=\chi_{\mu_-^\natural}$ in the case $k\not=0$, implies $\la^\natural=\mu^\natural$.

\item[(3)] Let $\G=\osp(2m|2n)$ with $m\ge 2$ and suppose that $k=1$. We have the following:
\begin{itemize}
\item[(i)] If $\la_n<m-1$, then
 $\chi_{\la^\natural}=\chi_{\mu^\natural}$, equivalently  $\chi_{\la^\natural}=\chi_{\mu_-^\natural}$,
implies $\la^\natural=\mu^\natural$.
\item[(ii)] If $\la_n\ge m-1$, then there are finitely many $\mu\in\mc{H}(n|m)$ satisfying $\chi_{\la^\natural}=\chi_{\mu^\natural}$, equivalently $\chi_{\la^\natural}=\chi_{\mu_-^\natural}$.
\end{itemize}
\end{itemize}
In addition, $(2)$ and $(3)$  hold if we replace $\la^\natural$, $\mu^\natural$, and $\mu_-^\natural$ therein by $\la_-^\natural$, $\mu_-^\natural$, and $\mu^\natural$, respectively.
\end{prop}
\begin{proof}
(1)
Suppose that $\G=\osp(2m+1|2n)$.  Let $\rho=\rho^{{\mf b}^{\texttt{st}}}$.
 Since $\chi_{\la^\natural}=\chi_{\mu^\natural}$, they have the same degree of atypicality $k$.  We consider $\la^\natural+\rho=\sum_{i=1}^n a_i\delta_i+\sum_{j=1}^m b_j\ep_j$ and $\mu^\natural+\rho=\sum_{i=1}^n c_i\delta_i+\sum_{j=1}^m d_j\ep_j$, and identify them with the following sequences (of half-integers)
\begin{equation}\label{aux:eqn1}
\begin{split}
\la^\natural+\rho&=(\,a_1,\ldots,a_n\,|\,b_1,\ldots,b_m\,),\\
\mu^\natural+\rho&=(\,c_1,\ldots,c_n\,|\,d_1,\ldots,d_m\,),
\end{split}
\end{equation}
respectively. Note that $W=(S_{n}\ltimes\Z_2^{n})\times (S_{m}\ltimes\Z_2^{m})$.

By Theorem \ref{thm:classification B}, we can find a set of isotropic odd roots $\{\,\delta_{i_r}-\epsilon_{j_r}\,|\,r=1,\ldots,k\,\}$ mutually orthogonal and also orthogonal to $\la^\natural+\rho$ for some $1\leq i_1,\ldots, i_k\leq n$ and $1\leq j_1,\ldots, j_k\leq m$. Removing all the entries $a_{i_r}$ and $b_{j_r}$ for $i=1,\ldots,k$, from $(\,a_1,\ldots,a_n\,|\,b_1,\ldots,b_m\,)$, we are left with an expression which we shall denote by $[\la^\natural+\rho]$. We do the same now for the sequence $\mu^\natural+\rho=(\,c_1,\ldots,c_n\,|\,d_1,\ldots,d_m\,)$. After removing all such entries in each of the two expressions in \eqref{aux:eqn1}, we are left with respective sequences
\begin{equation}\label{aux:eqn2}
\begin{split}
[\la^\natural+\rho]&:=[\,\tilde{a}_1,\ldots,\tilde{a}_{n-k}\,|\,\tilde{b}_1,\ldots,\tilde{b}_{m-k}\,],\\
[\mu^\natural+\rho]&:=[\,\tilde{c}_1,\ldots,\tilde{c}_{n-k}\,|\,\tilde{d}_1,\ldots,\tilde{d}_{m-k}\,].
\end{split}
\end{equation}
By \cite[Theorem 2.30]{CW} (see also \cite{K2, Sv2}), these two sequences coincide up to the action of $(S_{n-k}\ltimes\Z_2^{n-k})\times (S_{m-k}\ltimes\Z_2^{m-k})< W$  since $\chi_{\la^\natural}=\chi_{\mu^\natural}$. Now we note that $\tilde{b}_j$ and $\tilde{d}_j$ are positive since all $b_j$ and $d_j$ are positive, and $\tilde{a}_i$ and $\tilde{c}_i$ are also positive since $\la$ and $\mu$ are $(n|m)$-hook partitions (cf.~\eqref{eq:rhost}).
Hence the two subsequences $[\la^\natural+\rho]$ and $[\mu^\natural+\rho]$ coincide up to the action of  $S_{n-k}\times S_{m-k} < W$. This implies that $\la^\natural$ and $\mu^\natural$ have the same central character with respect to the Levi subalgebra $\gl(n|m)$ of $\G$ in \eqref{ABC:diagram:A1}. Now, we can apply \cite[Lemma 3.4]{CK} (see also \cite[Proposition 3.24]{CW}) to conclude that $\la^\natural=\mu^\natural$ or $\la=\mu$.\vskip 2mm

(2) Suppose that $\G=\osp(2m|2n)$.  Let $\rho=\rho^{{\mf b}^{\texttt{st}}}$. Here we have $W=(S_{n}\ltimes\Z_2^{n})\times (S_{m}\ltimes \overline{\Z_2^{m}})$, where $\overline{\Z_2^{m}}$ is the subset of $\Z_2^{m}$ with even number of $1$'s. We can assume that $k\ge 1$, since in the typical case it is well-known that the $L(\la^\natural)$ (respectively, $L(\la_-^\natural)$) is the only irreducible finite-dimensional module with the central character $\chi_{\la^\natural}$ (respectively, $\chi_{\la_-^\natural}$).

Suppose that $k>1$ and $\chi_{\la^\natural}=\chi_{\mu^\natural}$. By Theorem \ref{thm:classification D}(i), we have $(\la^\natural,\epsilon_m)=0$ and hence $\la^\natural=\la_-^\natural$. Now we can apply the same argument as in the case when $\G=\osp(2m+1|2n)$ to conclude $\la^\natural=\mu^\natural$.  \vskip 2mm

(3) Suppose that $\G=\osp(2m|2n)$ and $k=1$. Let $\rho=\rho^{{\mf b}^{\texttt{st}}}$.
For (i), we note that $\la_n<m-1$ implies that $(\la^\natural,\epsilon_m)=0$. Thus, we have $a_n<0$ and $b_m=0$ in \eqref{aux:eqn1}, and hence $\la^\natural + \rho=(\cdots -a|\cdots a\cdots 0)$ for some $a>0$. Suppose that $\chi_{\la^\natural}=\chi_{\mu^\natural}$.  Then there are two possibilities for $\mu^\natural+\rho$: (A) $\mu^\natural+\rho=(\cdots -c|\cdots c\cdots 0)$ with $(\mu^\natural+\rho,\delta_n-\epsilon_j)=0$, where $1\le j<m$ and $c>0$, or else (B) $\mu^\natural+\rho=(\cdots b\cdots|\cdots b)$ with $(\mu^\natural+\rho,\delta_i+\epsilon_m)=0$, where $1\le i\le n$ and $b\ge 0$. Here we have $b=0$ if and only if $i=n$. But (B) is not possible, since $b_m=\tilde{b}_{m-1}=0$ in \eqref{aux:eqn2}. Thus, (A) holds and we can apply the same argument as in  (2) to conclude that $\la^\natural=\mu^\natural$.

Finally suppose that $\la_n\ge m-1$. Thus, we have $a_n\ge 0$. Suppose that $\chi_{\la^\natural}=\chi_{\mu^\natural}$. By Theorem \ref{thm:classification D}(ii) and similar arguments as in (i), we know that the isotropic root $\beta$ orthogonal to $\la^\natural + \rho$ and the isotropic root $\gamma$ orthogonal to $\mu^\natural + \rho$ are necessarily of the form $\delta_i+\epsilon_m$ and $\delta_j+\ep_m$, respectively. Removing the non-negative entries $a_i,b_m$ and $c_j, d_m$ in \eqref{aux:eqn1} depending on the roots $\beta$ and $\gamma$, we are left with expressions $[\la^\natural+\rho]$ and $[\mu^\natural+\rho]$. As in (1), $[\la^\natural+\rho]$ and $[\mu^\natural+\rho]$ have only positive entries and hence coincide up to the action of $S_{n-1}\times S_{m-1} < W$. Hence there are only finitely many choices of $\mu^\natural$ or $\mu$. This settles (ii).
\end{proof}

\begin{rem}\label{rem:part32}
Let $\G=\osp(2m|2n)$ with $m\ge 2$ and $\la\in\mc{H}(n|m)$ with $\la_n\ge m-1$. Suppose that $L(\la^\natural)$ is tame so that the degree of atypicality is $1$. We can describe the finite set of highest weights in Theorem \ref{prop:central ch of tame}(3)(ii) explicitly. For this purpose, we let $\la^\natural+\rho=\sum_{i=1}^na_i\delta_i+\sum_{j=1}^mb_j\epsilon_j$ and suppose that $a_i=b_m$. Let $X=\{\,0,1,\ldots,b_{m-1}-1\,\}\setminus\{a_1,\ldots,\widehat{a_i},\ldots,a_n\}$, where \ $\widehat{}$ \, stands for omission, as usual. Now for each $x\in X$ we define $\la(x)\in\mc{H}(n|m)$ to be the hook partition that is uniquely determined by the property that $\la(x)^\natural+\rho=\sum_{i=1}^nc_i\delta_i+\sum_{j=1}^md_j\epsilon_j$, satisfying
\begin{align*}
&\{c_1,\ldots,c_n\}=\{a_1,\ldots,\widehat{a_i},\ldots,a_n\}\cup\{x\},\\
&\{d_1,\ldots,d_m\}=\{b_1,\ldots,b_{m-1},x\}.
\end{align*}
Then the set $\{\,\la(x)^\natural,\la(x)_-^\natural\,|\,x\in X\,\}$ is precisely the set of highest weights of tame $\G$-modules with central character $\chi_\la^{\natural}$. We note that $0\in X$ and $\la(0)^\natural$ is the minimum in the set $\{\,\la(x)^\natural,\la(x)_-^\natural\,|\,x\in X\,\}$ with respect to $\preccurlyeq$.
\end{rem}

\begin{example}
Consider $\mf{osp}(6|4)$.  Let $\la$ be the $(2|3)$-hook partition $(3,3,3,2,2,2,1)$ so that we have
\begin{align*}
\la^\natural+\rho=2\delta_1+\delta_2+7\epsilon_1+5\epsilon_2+\epsilon_3=(2,1|7,5,1).
\end{align*}
The set $X=\{0,1,2,3,4\}\setminus\{2\}=\{0,1,3,4\}$.  Thus, we have the following $4$ possibilities for $\la(x)^\natural+\rho$:
\begin{align*}
&\la(0)^\natural+\rho=(2,0|7,5,0),\text{ with }\la(0)=(3,2,2,2,2,2,1),\\
&\la(1)^\natural+\rho=(2,1|7,5,1),\text{ with }\la(1)=(3,3,3,2,2,2,1),\\
&\la(3)^\natural+\rho=(3,2|7,5,3),\text{ with }\la(3)=(4,4,3,3,3,2,1),\\
&\la(4)^\natural+\rho=(4,2|7,5,4),\text{ with }\la(4)=(5,4,3,3,3,3,1),
\end{align*}
with $\la^\natural=\la^\natural(1)$.
Hence, it follows that the weights in the set $\{\,\la(x)^\natural,\la(x)_-^\natural\,|\,x=0,1,3,4\,\}$ are precisely the highest weights of tame modules that have the same central character as $\la^\natural$. \end{example}

The following corollary is immediate from Proposition \ref{prop:central ch of tame}.

\begin{cor}
\begin{itemize}
\item[(1)]
Let $\G=\osp(2m+1|2n)$. Let $\mc{KW}$ be the Serre subcategory of the category of finite-dimensional $\G$-modules generated by the tame $\G$-modules. Then ${\mc{KW}}$ is semisimple.
\item[(2)] Let $\G=\osp(2m|2n)$. Let $\mc{KW}$ be the Serre subcategory of the category of finite-dimensional $\G$-modules generated by the tame $\G$-modules of $\mf b^{\texttt{st}}$-highest weights $\mu^\natural$ satisfying $\mu^\natural=\mu^\natural_-$. Then ${\mc{KW}}$ is semisimple.
\end{itemize}
\end{cor}

\begin{thm}\label{thm:KW:bottom}
Let $\chi$ be a central character of a finite-dimensional irreducible $\G$-module.
\begin{itemize}
\item[(1)] Let $\G=\osp(2m+1|2n)$. Then there exists a unique $\mu\in\mc{H}(n|m)$ such that
$L(\mu^\natural)$ is tame, $\chi=\chi_{\mu^\natural}$, and  $\mu^\natural \preccurlyeq \la^\natural$  for every $\la\in\mc{H}(n|m)$ with $\chi_{\la^\natural}=\chi$. Furthermore, any $\mu^\natural$ with $L(\mu^\natural)$ tame is the minimum in $\{\,\la^\natural\,|\,\la\in\mc{H}(n|m), \,\chi_{\la^\natural}=\chi_{\mu^\natural}\,\}$ with respect to $\preccurlyeq$.

\item[(2)] Let $\G=\osp(2m|2n)$ with $m\ge 2$. Then there exists a unique $\mu\in\mc{H}(n|m)$ such that $L(\mu^\natural)$ is tame,   $\chi=\chi_{\mu^\natural}$, and
$\mu^\natural \preccurlyeq \la^\natural$ (respectively, $\mu^\natural \preccurlyeq \la^\natural_-$)
for  every $\la\in\mc{H}(n|m)$ with $\chi_{\la^\natural}=\chi$  (respectively, $\chi_{\la_-^\natural}=\chi$). Furthermore, any $\mu^\natural$ with  $L(\mu^\natural)$ tame and  $\mu^\natural=\mu^\natural_-$  is the minimum of $\{\,\la^\natural,\la^\natural_-\,|\,\la\in\mc{H}(n|m),\, \chi_{\la^\natural}=\chi_{\mu^\natural}\,\}$ with respect to $\preccurlyeq$.
\end{itemize}
\end{thm}

\begin{proof}
%We assume that $k\geq 1$ since it is clear in the typical case.

(1) Suppose that $\G=\osp(2m+1|2n)$.
Let $\la\in\mc{H}(n|m)$ and consider the weight $\la^\natural$. We first claim that if $\la^\natural$ is not the highest weight of a tame module, then there exists $\mu\in \mc{H}(n|m)$ such that
\begin{equation*}\label{aux:eqn3}
\text{$L(\mu^\natural)$ is tame}, \quad \chi_{\mu^\natural}=\chi_{\la^\natural}\quad\text{and}\quad \mu^\natural \preccurlyeq \la^\natural.
\end{equation*}
We give an algorithm to obtain such a $\mu^\natural$ from a weight $\la^\natural$.
Let $\rho=\rho^{{\mf b}^{\texttt{st}}}$.
Consider $\la^\natural+\rho=\sum_{i=1}^na_i\delta_i+\sum_{j=1}^mb_j\epsilon_j$, which we write
\begin{align*}
\la^\natural+\rho=(\,a_1,\ldots,a_n\,|\,b_1,\ldots,b_m\,).
\end{align*}
Since $\la^\natural$ is not the highest weight of a tame module, there exists $b_j>0$ such that  $a_i=b_j$ for some $a_i$ but $-b_j\not\in\{a_1,\ldots,a_n\}$ by Theorem \ref{thm:classification B}. We assume that $b_j$ is minimal with such property. Now define
\begin{align}\label{formula:btilde}
\tilde{b}_j:=\min\left\{\,x\in \tfrac{1}{2}+\Z_+\,|\,x\le b_j,\ x\not=b_k,\ {-}x\not=a_i  \text{ for all $ k>j$ and $i$}\,\right\}.
\end{align}
Then we let $\mu\in\mc{H}(n|m)$ be the unique hook partition determined by  $\mu^\natural+\rho=\sum_{i=1}^nc_i\delta_i+\sum_{j=1}^md_j\epsilon_j$, where
\begin{align*}
\{c_1,\ldots,c_n\}&=\{a_1,\ldots,\widehat{a}_i,\ldots,a_n\}\cup\{-\tilde{b}_j\},\\
\{d_1,\ldots,d_m\}&=\{b_1,\ldots,\widehat{b}_j,\ldots,b_m\}\cup\{\tilde{b}_j\}.
\end{align*}
It is straightforward to verify that $\mu$ is well-defined. By definition, it is clear that $\chi_{\la^\natural}=\chi_{\mu^\natural}$. Also, we can check without difficulty that $\mu^\natural$ is strictly smaller than $\la^\natural$ with respect to $\preccurlyeq$, that is, $\lambda^\natural-\mu^\natural\in\Z_+\Phi^+_{\mf b^{\texttt{st}}}\setminus\{0\}$.
Now, if $L(\mu^\natural)$ is tame, then we are done.  If not, then we continue this process, which must stop in finitely many steps, to obtain a highest weight of a tame module. This proves our claim.
The uniqueness of $\mu^\natural$ in the claim follows from Proposition \ref{prop:central ch of tame}(1).

Now, suppose that $L(\mu^\natural)$ is tame, but $\mu^\natural$ is not a ``bottom'' of $\chi_{\mu^\natural}$, i.e., there exists $\la\in\mc{H}(n|m)$ such that $\la^\natural\precneqq \mu^\natural$ and $\chi_{\la^\natural}=\chi_{\mu^\natural}$. Then, by Proposition \ref{prop:central ch of tame}(1), $L(\la^\natural)$ is not tame, and hence, by our above discussion, we can find a $\gamma\in\mc{H}(n|m)$ such that $L(\gamma^\natural)$ is tame, $\chi_{\gamma^\natural}=\chi_{\la^\natural}$, and $\gamma^\natural\precneqq \la^\natural$. Thus, $\gamma^\natural\precneqq \mu^\natural$, and both give rise to tame modules and have the same central character.  This contradicts Proposition \ref{prop:central ch of tame}(1) and proves (1).

(2) Suppose that $\G=\osp(2m|2n)$.  Let $k$ be the degree of atypicality for $\chi$.
We assume that $k\geq 1$ since it is clear in the typical case.
Let $\la\in \mc{H}(n|m)$ be given such that $\chi=\chi_{\la^\natural}$.
Then by similar arguments as in (1), we can show that there exists a $\mu\in \mc{H}(n|m)$ such that $L(\mu^\natural)$ is tame, $\chi=\chi_{\mu^\natural}$, and $\mu^\natural \preccurlyeq \la^\natural$. One difference here is that one has $\tilde{b}_j\in\Z_+$  in \eqref{formula:btilde}. Note that if $k>1$, then we have $\mu_{n+1}<m$ by Theorem \ref{thm:classification D}, and if $k=1$, then we may also assume by Proposition \ref{prop:central ch of tame}(3) and Remark \ref{rem:part32} that $\mu_{n+1}<m$.
This implies that $\mu^\natural$ is a unique highest weight such that $L(\mu^\natural)$ is tame, $\chi=\chi_{\mu^\natural}$, and $\mu^\natural \preccurlyeq \la^\natural$ for any $\la\in \mc{H}(n|m)$ with $\chi=\chi_{\la^\natural}$ by Proposition \ref{prop:central ch of tame}.

Now, let $\la\in \mc{H}(n|m)$ be given such that $\chi=\chi_{\la_-^\natural}$. Since $\mu^\natural=\mu^\natural_-$, we have $\chi_{\mu^\natural}=\chi_{\mu^\natural_-}=\chi_{\la^\natural}$, which implies that $\mu^\natural \preccurlyeq \la^\natural$ and hence, applying $\sigma$, we have $\mu^\natural \preccurlyeq \la^\natural_-$. The last statement in (2) is proved analogously as in (1), now using Proposition \ref{prop:central ch of tame}(2)--(3) and Remark \ref{rem:part32}.
\end{proof}

\begin{example}
We illustrate the algorithm in Theorem \ref{thm:KW:bottom} with an example.
Consider the $\osp(7|6)$ and $\la=(6,6,5,2,1,1)$ so that we have
\begin{align*}
\la^\natural+\rho=\left(\frac{11}{2},\frac{9}{2},\frac{5}{2}\,\Big\vert\,\frac{11}{2},\frac{5}{2},\frac{1}{2}\right).
\end{align*}
Applying the process given in the proof of Theorem \ref{thm:KW:bottom}(1) once gives the weight
\begin{align*}
\nu^\natural+\rho=\left(\frac{11}{2},\frac{9}{2},-\frac{3}{2}\,\Big|\frac{11}{2},\frac{3}{2},\frac{1}{2}\right),
\end{align*}
with $\nu=(6,6,1,1,1,1)$. Now, applying the process again to $\nu^\natural$, we get
\begin{align*}
\gamma^\natural+\rho=\left(\frac{9}{2},-\frac{3}{2},-\frac{5}{2}\Big|\frac{5}{2},\frac{3}{2},\frac{1}{2}\right),
\end{align*}
with $\gamma=(5,0,0,0,0,0)$, which is a highest weight of a tame module.
\end{example}

\begin{rem}
It can be shown that analogous statements as in Theorem \ref{thm:KW:bottom} hold also for tame modules over the exceptional Lie superalgebras $D(2|1,\alpha)$, $G(3)$, and $F(3|1)$. Thus, we conclude that the ``bottoms'' of each finite-dimensional block for type II Lie superalgebras are highest weights of tame modules.
\end{rem}

\begin{prop}\label{prop:aux01}
Let $\la\in\mc{H}(n|m)$ such that $L(\la^\natural)$ is tame of degree of atypicality $k\geq 1$. Let ${\mf b}$ and ${\mf l}$ be the   Borel and Levi subalgebras  in \eqref{eq:borel levi odd-1}, \eqref{eq:borel levi odd-2}, and \eqref{eq:borel levi odd-3}, respectively. Then  the parabolic subalgebra $\mf{p}={\mf l}+{\mf b}$ is admissible for $\chi_{\la^{\mf b}}$, and
{$$\left(\la^{\mf b}+\rho^{\mf b},\tfrac{\beta}{(\beta,\beta)}\right)>0, \quad\quad \text{for } \beta\in\Phi^+_{\mf b}(\mf u_\even).$$}
\end{prop}

\begin{proof}
Let $\rho=\rho^{\mf b^{\texttt{st}}}$ and consider $\la^\natural+\rho=\sum_{i=1}^n a_i\delta_i+\sum_{j=1}^m b_j\ep_j$, where the degree of atypicality of $\la^\natural$ is $k\ge 1$. Recall that ${\mf u}$ is the nilradical of ${\mf p}$.

Suppose first that $\G=\osp(2m+1|2n)$. In this case, we have
%\begin{equation*}
%\Phi^+_{\mf b}(\mf u_\even)=
%\left\{\,
%\delta_i\pm\delta_j, 2\delta_p, \epsilon_s\pm\epsilon_t,\epsilon_q
%\,\Bigg\vert\,
%\begin{array}{l}
%\text{$1\leq i <j\leq n$ with $i\leq n-k$,} \\
%\text{$1\leq p\leq n-k$, $1\leq q\leq m-k$,}       \\
%\text{$1\leq s<t\leq m$ with {$s\leq m-k$}}       \\
%\end{array}
%\,\right\},
%\end{equation*}
\begin{equation*}
\Phi^+_{\mf b}(\mf u_\even)=\{\,\delta_i\pm\delta_j, 2\delta_p, \epsilon_s\pm\epsilon_t,\epsilon_q\,\},
\end{equation*}
where the indices are over $1\leq i <j\leq n$ with $i\leq n-k$, $1\leq p\leq n-k$, and $1\leq s<t\leq m$ with $s\leq m-k$, $1\leq q\leq m-k$.
By Lemma \ref{lem:htwt:bodd-1}(1), it is clear that we have
\begin{align}\label{eq:aux002}
\left(\la^{\mf b}+\rho^{\mf b},\tfrac{\beta}{(\beta,\beta)}\right)>0,  \quad\quad \text{ for $\beta\in\Phi^+_{\mf b}(\mf u_\even)$}.
\end{align}
Given a $\mu\in\mc{H}(n|m)$ such that $\chi_{\mu^\natural}=\chi_{\la^\natural}$.
The algorithm in the proof of Theorem \ref{thm:KW:bottom}(1) implies that starting from $\mu^\natural$ we eventually get to $\la^\natural$.
{ From this we see that $\la\subseteq\mu$ as partitions, and thus $(\mu^{\mf b}+\rho^{\mf b},\tfrac{\beta}{(\beta,\beta)})\ge (\la^{\mf b}+\rho^{\mf b},\tfrac{\beta}{(\beta,\beta)})>0$, for $\beta=\delta_i+\delta_j, 2\delta_p, \epsilon_s+\epsilon_t, \ep_q\in \Phi^+_{\mf b}(\mf u_\even)$ by \eqref{eq:aux004} and Lemma \ref{lem:htwt:bodd-1}(1). It is also clear that $(\mu^{\mf b}+\rho^{\mf b},\tfrac{\beta}{(\beta,\beta)})\ge 0$ for the other roots in $\Phi^+_{\mf b}(\mf u_\even)$.} This together with \eqref{eq:aux002} implies that $\mf p$ is admissible for $\chi_{\la^\mf b}$.

Now suppose that $\G=\osp(2m|2n)$.

Consider first the case when $\la^\natural=\la^\natural_-$. In this case, we have
%\begin{equation*}\label{eq:roots for u}
%\Phi^+_{\mf b}(\mf u_\even)=
%\left\{\,\delta_i\pm\delta_j, 2\delta_p, \epsilon_s\pm\epsilon_t\,\Bigg\vert\,
%\begin{array}{l}
%\text{$1\leq i <j\leq n$ with $i\leq n-k$,} \\
%\text{$1\leq p\leq n-k$,}       \\
%\text{$1\leq s<t\leq m$ with {$s\leq m-k-e(\la)$}}       \\
%\end{array}
%\,\right\},
%\end{equation*}
\begin{equation*}\label{eq:roots for u}
\Phi^+_{\mf b}(\mf u_\even)=\{\,\delta_i\pm\delta_j, 2\delta_p, \epsilon_s\pm\epsilon_t\,\},
\end{equation*}
where the indices are over $1\leq i <j\leq n$ with $i\leq n-k$, $1\leq p\leq n-k$, and $1\leq s<t\leq m$ with {$s\leq m-k-e(\la)$}. Recall that $e(\la)$ is given in \eqref{eq:e(lambda)}. Again, by Lemma \ref{lem:htwt:bodd-2}(1) it is clear that \eqref{eq:aux002} holds in this case. Let $\mu\in\mc{H}(n|m)$ such that $\chi_{\mu^\natural}=\chi_{\la^\natural}$. Then a similar argument as in the case of $\G=\osp(2m+1|2n)$ shows that $\la\subseteq\mu$ as partitions, and we again arrive that {$(\mu^{\mf b}+\rho^{\mf b},\tfrac{\beta}{(\beta,\beta)})\ge (\la^{\mf b}+\rho^{\mf b},\tfrac{\beta}{(\beta,\beta)})>0$, for $\beta=\delta_i+\delta_j, 2\delta_p, \ep_s+\ep_t\in\Phi^+_{\mf b}(\mf u_\even)$.
It is also clear that $(\mu^{\mf b}+\rho^{\mf b},\tfrac{\beta}{(\beta,\beta)})\ge 0$ for the other roots in $\Phi^+_{\mf b}(\mf u_\even)$.}
Now suppose that $\chi_{\nu_-^\natural}=\chi_{\la^\natural}$, for some $\nu\in\mc{H}(n|m)$. Since  $\chi_{\nu^\natural}=\chi_{\la^\natural}$,  we get $\la\subseteq\nu$ again and hence {$(\nu^{\mf b}+\rho^{\mf b},\tfrac{\beta}{(\beta,\beta)})\ge 0$}, for $\beta\in\Phi^+_{\mf b}(\mf u_\even)$. Note that both Dynkin diagrams of $\mf b$ and $\mf b^{\texttt{st}}$ end with an $\epsilon$. Thus, by \eqref{eq:aux001}, we see that $(\nu_-^\mf b+\rho^\mf b,\sigma(\beta))=(\nu^\mf b+\rho^\mf b,\beta)$, for $\beta\in\Phi^+_{\mf b}(\mf u_\even)$. This implies that $\mf p$ is admissible for $\chi_{\la^\mf b}$.

Now suppose that $\la^\natural\not=\la_-^\natural$. As in \eqref{aux:eqn1} we write $\la^\natural+\rho$ as a sequence
\begin{align*}
\la^\natural+\rho =(\ldots a_i\ldots|\ldots b_{m-1},b_m),
\end{align*}
and we have $a_i=b_m=b>0$ for some $i$. Recall that the Borel subalgebra $\mf{b}$ has the Dynkin diagram \eqref{ABC:diagram:D7}.
The sequence of odd reflections or the corresponding odd simple roots  that transforms $\mf b^{\texttt{st}}$ to $\mf b$ is as follows:
\begin{align}\label{seq:odd:ref}
\begin{split}
&\delta_n-\epsilon_1,\delta_{n-1}-\epsilon_1,\ldots,\delta_i-\epsilon_1,\\
&\delta_n-\epsilon_2,\delta_{n-1}-\epsilon_2,\ldots,\delta_i-\epsilon_2,\\
&\qquad\vdots\qquad\vdots\qquad\vdots\\
&\delta_n-\epsilon_{m-1},\delta_{n-1}-\epsilon_{m-1},\ldots,\delta_i-\epsilon_{m-1},\\
&\delta_n+\epsilon_m,\delta_{n-1}+\epsilon_m,\ldots,\delta_{i+1}+\epsilon_m.
\end{split}
\end{align}
Here we apply the odd reflections successively in the above order starting from $\delta_n-\ep_1$.
By \eqref{eq:odd reflection rho} we have $\la^\mf b+\rho^\mf b=\la^\natural+\rho$. Recall that $\mf{l}=\h+\mf{sl}(1|1)$ and $\mf{p}=\mf{l}+\mf{b}$.  In this case
$\Phi_{\mf b}^+(\mf u_\even)$ consists of all even positive roots (of $\mf b^{\texttt{st}}$). It is clear that we have $(\la^{\mf b}+\rho^{\mf b},\tfrac{\beta}{(\beta,\beta)})>0$, for $\beta\in\Phi^+_{\mf b}(\mf u_\even)$ and thus \eqref{eq:aux002} holds in this case.

Let $\mu\in\mc H(n|m)$ such that $\chi_{\mu^\natural}=\chi_{\la^\natural}$. Since the degree of atypicality equals one, one can derive that $\mu^\natural$ is such that either $L(\mu^\natural)$ is tame, or else $\mu^\natural+\rho$ is of the form
\begin{align}\label{eq:aux003}
\mu^\natural+\rho =(\ldots c\ldots|\ldots c\ldots b_{m-1}),
\end{align}
with $c>b_m>0$.

If $\mu^\natural+\rho$ is of the form \eqref{eq:aux003}, then we can apply the procedure in Theorem \ref{thm:KW:bottom} to get $\la\subseteq\mu$ as partitions, and again it follows that $(\mu^{\mf b}+\rho^{\mf b},\tfrac{\beta}{(\beta,\beta)})\ge 0$, for $\beta\in\Phi^+_{\mf b}(\mf u_\even)$.

Now, suppose that $L(\mu^\natural)$ is tame. There are two cases for $\mu^\natural+\rho$:
\begin{align}
\mu^\natural+\rho =(\ldots c\ldots\ldots|\ldots b_{m-1},c),\quad c>b,\label{eq:aux111}\\
\mu^\natural+\rho =(\ldots\ldots  d\ldots|\ldots b_{m-1},d),\quad d<b.\label{eq:aux112}
\end{align}
Recall the sequence of odd reflections \eqref{seq:odd:ref} that transforms $\mf b^{\texttt{st}}$ to $\mf b$.
Using again \eqref{eq:odd reflection rho} we see that,
in the case of \eqref{eq:aux111}, applying this sequence of odd reflections does not change shifted weight and hence $\mu^\mf b+\rho^{\mf b}=\mu^\natural+\rho$. In the case of \eqref{eq:aux112} we have $\mu^\mf b+\rho^{\mf b}=\mu^\natural+\rho +\sum_{s=i+1}^{k}\eta_s(\delta_s+\epsilon_m)$, with $\eta_s\in\{0,1\}$, for some $k>i$. Thus, in either case, we have $(\mu^{\mf b}+\rho^{\mf b},\tfrac{\beta}{(\beta,\beta)})\ge 0$, for $\beta\in\Phi^+_{\mf b}(\mf u_\even)$.

Finally, suppose that $\mu^\natural\not=\mu^\natural_-$, and $\chi_{\la^\natural}=\chi_{\mu^\natural_-}$.  We have the following possibilities for $\mu^\natural_-+\rho$:
\begin{align}
\mu_-^\natural+\rho &=(\ldots c\ldots\ldots|\ldots c\ldots -b'),\quad c>b'>0,\label{eq:aux113}\\
\mu_-^\natural+\rho &=(\ldots\ldots  d\ldots|\ldots\ldots b_{m-1},-d),\quad d>0\label{eq:aux114}.
\end{align}
In the case of \eqref{eq:aux113} we have $b'=b_{m-1}$, and hence $c>b_m$. In the case of \eqref{eq:aux114} $L(\mu_-^\natural)$ is tame. In either case, applying the same sequence of odd reflections above does not change the shifted weights, which implies that $(\mu^{\mf b}_-+\rho^{\mf b},\tfrac{\beta}{(\beta,\beta)})\ge 0$, for $\beta\in\Phi^+_{\mf b}(\mf u_\even)$.  Thus, $\mf p$ is admissible for $\chi_{\la^\mf b}$.
\end{proof}

\section{A proof of the Kac-Wakimoto conjecture}\label{sec:proof:KW}

In this section, we prove Conjecture \ref{conj:KW} for the ortho-symplectic Lie superalgebras. The following is a  particular case of \cite{Gor}.

\begin{prop}\label{prop:KW:trivial}
Conjecture \ref{conj:KW} holds for the trivial module over $\mf g=\osp(2k+1|2k)$, $\osp(2k|2k)$, or {$\osp(2k+2|2k)$}, for $k\geq 1$ and $\mf b=\mf b^{\texttt{odd}}$. That is, we have
\begin{align*}
1=\frac{1}{j}\,D_{\mf b}^{-1}%\frac{D_{\mf b,\bar{1}}}{D_{\mf b,\bar{0}}}
\sum_{w\in W}(-1)^{\ell(w)} w\left(  \frac{e^{\rho^{\mf b}}}{\prod_{\beta\in T}(1+e^{-\beta})} \right),
\end{align*}
where
\begin{equation*}
\begin{split}
j&=
\begin{cases}
k!2^k, & \text{for $\G=\osp(2k+1|2k),\ {\osp(2k+2|2k)}$},\\
k!2^{k-1}, & \text{for $\G=\osp(2k|2k)$},\\
\end{cases}\\
T&=
\begin{cases}
\{\,\ep_i-\delta_j\,|\,1\leq i,j\leq k\,\}, & \text{for $\G=\osp(2k+1|2k)$},\\
\{\,\delta_i-\ep_j\,|\,1\leq i,j\leq k\,\}, & \text{for $\G=\osp(2k|2k),\ {\osp(2k+2|2k)}$}.
\end{cases}
\end{split}
\end{equation*}
\end{prop}

\begin{proof}
%Here is a sketch of a proof.
{Let $\G=\osp(2k+1|2k)$, $\osp(2k|2k)$, or $\osp(2k+2|2k)$.
By \cite[Propositions 3.1 and 8.1]{SV2}, with certain $\mf p$ and $\mf l=\gl(k|k)$ (for $\G\neq \osp(2k+2|2k)$) or $\gl(k+1|k)$ (for $\G=\osp(2k+2|2k)$) with respect to the Borel subalgebra ${\mf b}={\mf b}^{\texttt{st}}$}, we have
\begin{equation}\label{eq:denom-1}
{\rm ch}\mc E({\rm Ind}_{\mf p}^{\G}M)=
\begin{cases}
2^k,\quad&\text{if }\G=\osp(2k+1|2k), {\osp(2k+2|2k)}, \\
2^{k-1},\quad&\text{if }\G=\osp(2k|2k),
\end{cases}
\end{equation}
where $M$ is the trivial ${\mf l}$-module. We note that in some special low rank cases, \eqref{eq:denom-1} was noticed in \cite[Examples 2.5 and 2.6]{CW1}. Now, by Remark \ref{form:Euler:1} we also get the same result if we choose the same $\mf p$ and $\mf l$ but now with respect to the Borel subalgebra ${\mf b}=\mf b^{\texttt{odd}}$.

By Theorem \ref{thm:CHR2}, the trivial $\mf l$-module $M$ satisfies the formula \eqref{formula:kw} with respect to $\underline{\mf b}={\mf b}\cap {\mf l}$ and  with the leading coefficient equals to $\frac{1}{k!}$. Now applying the same argument as in the proof of Proposition \ref{prop:Euler=KW} to $\mc E({\rm Ind}_{\mf p}^{\G}M)$ we get that
\begin{align}\label{eq:denom-2}
{\rm ch}\mc E({\rm Ind}_{\mf p}^{\G}M)=\frac{1}{k!}%\frac{D_{\mf b,\bar{1}}}{D_{\mf b,\bar{0}}}
D_{\mf b}^{-1}\sum_{w\in W}(-1)^{\ell(w)} w\left(  \frac{e^{\rho^{\mf b}}}{\prod_{\beta\in T}(1+e^{-\beta})} \right),
\end{align}
where $T=T_{0^{\mf b}}=T_{0^{\underline{\mf b}}}$. Combining \eqref{eq:denom-1} and \eqref{eq:denom-2}, we obtain the desired identity and hence the character formula \eqref{formula:kw} for the trivial $\G$-module.
\end{proof}

\begin{thm}\label{thm:main result}
Let $\G=\osp(\ell|2n)$ with $\ell=2m$ or $\ell=2m+1$ and let $\la\in\mc{H}(n|m)$ such that $L(\la^\natural)$ is tame of degree of atypicality $k\geq 1$.
\begin{itemize}
\item[(1)]
If $L(\la^\natural)$ is tame with respect to a Borel subalgebra $\mf b$ with a distinguished set $T_{\la^{\mf b}}\subseteq\Pi_{\mf b}$, then
\begin{align*}\label{KW:formula:B}
{\rm ch}L(\la^\natural)=\frac{1}{j_\la}%\frac{D_{\mf b,\bar{1}}}{D_{\mf b,\bar{0}}}
D_{\mf b}^{-1}\sum_{w\in W}(-1)^{\ell(w)} w\left(  \frac{e^{\la^\mf b+\rho^{\mf b}}}{\prod_{\beta\in T_{\la^\mf b}}(1+e^{-\beta})} \right),
\end{align*}
where
\begin{equation*}
j_\la=
\begin{cases}
k!2^k, & \text{if $\ell=2m+1$},\\
k!2^{k-1+e(\la)},& \text{if $\ell=2m$ and $\la_{n+1}<m$},\\
1, & \text{if $\ell=2m$ and $\la_{n+1}=m$}.
\end{cases}
\end{equation*}

\item[(2)] When $\ell=2m$, we have in addition the following character formula for $L(\la_-^\natural)$, which is tame with respect to $\sigma({\mf b})$:
\begin{align*}
{\rm ch}L(\la_-^\natural)
=\frac{1}{j_\la}%\frac{D_{\mf b,\bar{1}}}{D_{\mf b,\bar{0}}}
D_{\sigma(\mf b)}^{-1}\sum_{w\in W}(-1)^{\ell(w)} w\left(  \frac{e^{\la_-^{\sigma(\mf b)}+\rho^{\sigma(\mf b)}}}{\prod_{\beta\in T_{\la_-^{\sigma(\mf b)}}}(1+e^{-\beta})} \right),
\end{align*}
where $\sigma$ is defined in Section \ref{sec:Borels D} and $T_{\la_-^{\sigma(\mf b)}}=\sigma(T_{\la^\mf b})$.
\end{itemize}
\end{thm}

\begin{proof}
Case (1): Let $\mf b$ and ${\mf l}$ be as in \eqref{eq:borel levi odd-1}, \eqref{eq:borel levi odd-2}, and \eqref{eq:borel levi odd-3}.

By Lemmas \ref{lem:htwt:bodd-1}(3) and \ref{lem:htwt:bodd-2}(3), $L(\la^\natural)$ is tame with respect to $\mf b$, and ${\mf l}$ contains $T_{\la^{\mf b}}$.  It follows from Proposition \ref{prop:aux01} that the parabolic subalgebra $\mf p$ with Levi subalgebra $\mf l$ is admissible for $\chi_{\la^\mf b}$, and $(\la^{\mf b}+\rho^{\mf b},\tfrac{\beta}{(\beta,\beta)})>0$ for $\beta\in\Phi^+_{\mf b}(\mf u_\even)$. Therefore, we have by  Lemma \ref{lem:GS}
\begin{equation}\label{eq:p typicality of tame}
{\rm ch}\mc E({\rm Ind}_{\mf p}^\G L(\mf l,\la^\mf b))={\rm ch}L(\G,\la^{\mf b}).
\end{equation}

On the other hand, since $L({\mf l},\la^{\mf b})$ is the trivial module by Lemmas \ref{lem:htwt:bodd-1}(2) and \ref{lem:htwt:bodd-2}(2), we have by Proposition \ref{prop:KW:trivial}
\begin{equation*}
{\rm ch}L({\mf l},\la^{\mf b})=\frac{1}{j_\la}D_{\underline{\mf b}}^{-1}
\sum_{w\in W_{\mf l}}(-1)^{\ell(w)} w\left(  \frac{e^{\la^{\mf b}+\rho^{\underline{\mf b}}}}{\prod_{\beta\in T_{\la^{\mf b}}}(1+e^{-\beta})} \right),
\end{equation*}
where $\underline{\mf b}={\mf b}\cap{\mf l}$, $T_{\la^{\mf b}}=T$ is as in Proposition \ref{prop:KW:trivial}, and $j_\la=k!2^k$ and $k!2^{k-1+e(\la)}$ in the case of \eqref{eq:borel levi odd-1} and \eqref{eq:borel levi odd-2}, respectively. In the case of \eqref{eq:borel levi odd-3}, we see by Theorem \ref{thm:CHR2} that ${\rm ch}L({\mf l},\la^{\mf b})$ also has the form \eqref{formula:kw}  with $T_{\la^{\mf b}}=\{\,\delta_i+\ep_m\,\}$ and $j_\la=1$.
Applying Proposition \ref{prop:Euler=KW} to \eqref{eq:p typicality of tame}, we obtain the formula
\begin{equation}\label{eq:KW formula for b_odd}
{\rm ch}L(\G,\la^{\mf b})=
\frac{1}{j_\la}%\frac{D_{\mf b,\bar{1}}}{D_{\mf b,\bar{0}}}
D_{\mf b}^{-1}\sum_{w\in W}(-1)^{\ell(w)} w\left(  \frac{e^{\la^\mf b+\rho^{\mf b}}}{\prod_{\beta\in T_{\la^\mf b}}(1+e^{-\beta})} \right),
\end{equation}
which proves (1) with respect to ${\mf b}$.

Assume that $\G=\osp(2m|2n)$. By Lemma \ref{lem:sigma tame} $L(\la^\natural_-)$ is also tame with respect to $\sigma(\mf b)$ with a distinguished set $T_{\la_-^{\sigma(\mf b)}}=\sigma(T_{\la^{\mf b}})$. Combining \eqref{eq:sigma twist} and \eqref{eq:KW formula for b_odd}, we compute
\begin{equation}\label{eq:KW formula under sigma}
\begin{split}
%{\rm ch}L({\mf g},\la^\natural_-)&={\rm ch}L({\mf g},\la^\natural)^\sigma
{\rm ch}L(\G,\la_-^{\sigma(\mf b)})&={\rm ch}L({\mf g},\la^{\mf b})^\sigma
=\sigma({\rm ch}L({\mf g},\la^{\mf b}))\\
&=\sigma\left(\frac{1}{j_\la}
D_{\mf b}^{-1}\sum_{w\in W}(-1)^{\ell(w)} w\left(  \frac{e^{\la^\mf b+\rho^{\mf b}}}{\prod_{\beta\in T_{\la^\mf b}}(1+e^{-\beta})} \right)\right)\\
&=\frac{1}{j_\la}%\frac{D_{\mf b,\bar{1}}}{D_{\mf b,\bar{0}}}
D_{\sigma(\mf b)}^{-1}\sum_{w\in W}(-1)^{\ell(w)} w\left(  \frac{e^{\la_-^{\sigma(\mf b)}+\rho^{\sigma(\mf b)}}}{\prod_{\beta\in \sigma(T_{\la^\mf b})}(1+e^{-\beta})} \right),
\end{split}
\end{equation}
which is precisely the Kac-Wakimoto formula for $L(\la^\natural_-)$ with respect to $\sigma({\mf b})$. In the last equality above we have used the fact that $W\sigma=\sigma W$, since the Weyl group of type $D$ is a subgroup of index two in the Weyl group of type $B$ of the same rank. This proves (2) for ${\mf b}$.

%Note that  $T_{\la^{\mf b}}=T$ by Lemmas \ref{lem:htwt:bodd-1}(3) and \ref{lem:htwt:bodd-2}(3).

Case (2): Let ${\mf b'}$ be any Borel subalgebra of $\osp(\ell|2n)$ such that $L(\la^\natural)$ is tame with respect to ${\mf b'}$ with a distinguished set $T_{\la^{\mf b'}}$.

Suppose that $\ell=2m+1$ or $\ell=2m$ with $\la_{n+1}<m$.
From the proof of Theorems \ref{thm:classification B} and \ref{thm:classification D}(i),
we see that $T_{\la^{\mf b'}}$ consists of simple roots of the form $\pm(\delta_i-\epsilon_j)$ and $s({\mf b'})=1$ when $\ell=2m$, and hence the Borel subalgebra for $\gl(n|m)$ corresponding to $\mf b'$ contains $T_{\la^{\mf b'}}$. Now applying the exact same argument as in \cite[Propositions 39 and 40]{CHR}, we have
\begin{equation*}
\sum_{w\in W}(-1)^{\ell(w)} w\left(  \frac{e^{\la^\mf b+\rho^{\mf b}}}{\prod_{\beta\in T_{\la^\mf b}}(1+e^{-\beta})} \right)=
\sum_{w\in W}(-1)^{\ell(w)} w\left(  \frac{e^{\la^{\mf b'}+\rho^{\mf b'}}}{\prod_{\beta\in T_{\la^{\mf b'}}}(1+e^{-\beta})} \right),
\end{equation*}
where ${\mf b}$ is as in Case (1).
Since $D_{\mf b}=D_{\mf b'}$, the formula \eqref{eq:KW formula for b_odd} also holds for ${\mf b'}$.

Next, suppose that $\ell=2m$ and $\la_{n+1}=m$.
From the proof of Theorem \ref{thm:classification D}(ii), we see that $T_{\la^{\mf b'}}=\{\, \pm(\delta_j+\epsilon_m)\,\}$ for some $j$. We have $T_{\la_-^{\sigma(\mf b')}}=\sigma(T_{\la^{\mf b'}})=\{\, \pm(\delta_j-\epsilon_m)\,\}$, and $s(\sigma(\mf b'))=1$ when the $\epsilon\delta$-sequence for $\mf b'$ ends with $\delta$.
Note that $s(\sigma(\mf b'))=s(\sigma(\mf b))=1$, where ${\mf b}$ is as in Case (1), and both $\la_-^{\sigma(\mf b')}+\rho^{\sigma(\mf b')}$ and $\la_-^{\sigma(\mf b)}+\rho^{\sigma(\mf b)}$ can be viewed as a shifted highest weight for a finite-dimensional irreducible $\gl(n|m)$-module (not a polynomial module) with respect to the Borel subalgebras corresponding to $\sigma(\mf b')$ and $\sigma(\mf b)$, respectively.
So we have
\begin{align*}
\sigma({\rm ch}L({\mf g},\la^{\mf b'}))
&=\sigma({\rm ch}L({\mf g},\la^{\mf b}))\\
& = \frac{1}{j_\la}D_{\sigma(\mf b)}^{-1}
\sum_{w\in W}(-1)^{\ell(w)} w\left(  \frac{e^{\la_-^{\sigma(\mf b)}+\rho^{\sigma(\mf b)}}}{\prod_{\beta\in T_{\la_-^{\sigma(\mf b)}}}(1+e^{-\beta})} \right)\\
& = \frac{1}{j_\la}D_{\sigma(\mf b')}^{-1}
\sum_{w\in W}(-1)^{\ell(w)} w\left(  \frac{e^{\la_-^{\sigma(\mf b')}+\rho^{\sigma(\mf b')}}}{\prod_{\beta\in T_{\la_-^{\sigma(\mf b')}}}(1+e^{-\beta})} \right)\\
&=
\sigma\left(\frac{1}{j_\la}D_{\mf b'}^{-1}
\sum_{w\in W}(-1)^{\ell(w)} w\left(\frac{e^{\la^{\mf b'}+\rho^{\mf b'}}}{\prod_{\beta\in T_{\la^{\mf b'}}}(1+e^{-\beta})} \right)\right),\\
\end{align*}
where the second equality is from \eqref{eq:KW formula under sigma}, and the third equality is obtained by using the same arguments as in \cite[Propositions 39 and 40]{CHR} again.
Applying $\sigma$ on both sides, we obtain the formula for ${\rm ch}L({\mf g},\la^{\mf b'})$, which proves (1).
The proof of (2) is the same as in \eqref{eq:KW formula under sigma}.
\end{proof}

We conclude with an interpretation of the Kac-Wakimoto character formula as certain deformed Jacobi polynomials.  Let $e$ be an indeterminate as before and set $x_i=e^{\ep_i}$ and $y_j=e^{\delta_j}$, for $1\le i\le m$ and $1\le j\le n$. Let ${u}=\{\,{x_i+x_i^{-1}}\,|\,i=1,\ldots,m\,\}$ and ${v}=\{\,{y_j+y_j^{-1}}\,|\,j=1,\ldots,n\,\}$. Recall the super Jacobi polynomials $SJ_{\la}(u, v;k,p,q)$ of Sergeev and Veselov for $\la\in \mc{H}(m|n)$, which were introduced in \cite{SV1} in the study of deformed Calogero-Moser systems. In \cite{SV2}, specialized limit versions $SJ_{\la}(u, v;-1,-1, 0)$ and $SJ_{\la}(u, v;-1,0, 0)$ have been studied in connection with Euler characteristic for $\osp(\ell|2n)$. As a consequence of Theorem \ref{thm:main result}, we obtain the following interpretation of a special class of specialized super Jacobi polynomials as irreducible super characters of tame modules over the ortho-symplectic Lie superalgebras.

\begin{cor}\label{cor:super:Jacobi}
Let $\G=\osp(\ell|2n)$ with $\ell=2m$ or $\ell=2m+1$ and $\la\in\mc{H}(n|m)$ such that $L(\la^\natural)$ is tame.
\begin{itemize}
\item[(1)] Suppose that $\ell=2m+1$. Then the super character for the $\G$-module $L(\la^\natural)$ is given by the specialized super Jacobi polynomial $SJ_{\la'}(u, v;-1,-1, 0)$, up to a sign.

\item[(2)] Suppose that $\ell=2m$ and $\la^\natural=\la_-^\natural$. Then  the super character for the $\G$-module $L(\la^\natural)$ is given by the specialized super Jacobi polynomial $SJ_{\la'}(u, v;-1,0, 0)$, {up to a sign}.

\item[(3)] Suppose that $\ell=2m$ and $\la^\natural\not=\la_-^\natural$. Then  the super character for the direct sum of $\G$-modules $L(\la^\natural)\oplus L(\la_-^\natural)$ is given by the specialized super Jacobi polynomial $SJ_{\la'}(u, v;-1,0, 0)$, up to a sign.
\end{itemize}
\end{cor}

\begin{proof} Consider the cases of $\la$ as in (1) and (2). According to \cite[Theorems 7.2 and 8.7, and Remark 8.8]{SV2}, the specialized Jacobi polynomials $SJ_{\la'}(u, v;-1,-1, 0)$ and $SJ_{\la'}(u, v;-1,0, 0)$ in these cases are equal, up to  a sign, to the super characters of the virtual modules $\mc E({\rm Ind}_{\mf{p}}^\G L(\mf l,\la^{\mf b}))$ with respect to the Borel subalgebras $\mf b$ in Theorem \ref{thm:main result} (see  the formulas (41) and (53) in loc.~cit.~for explicit forms). Since $L(\la^\natural)$ is tame, we have $\mc E({\rm Ind}_{\mf{p}}^\G L(\mf l,\la^{\mf b}))=L(\G,\la^{\mf b})$ by Theorem \ref{thm:main result}, and hence (1) and (2) follow.

The proof of (3) is also similar using the formula (54) in \cite[Theorem 8.7]{SV2}.
\end{proof}

\end{document}